\documentclass{article}

\usepackage{geometry}
\usepackage[utf8]{inputenc}
\usepackage{amsfonts}
\usepackage{amsmath}
\usepackage{amssymb}
\usepackage{amsthm}
\usepackage{mathtools}
\usepackage{tikz}
\usepackage{graphicx}
\usepackage{braket}
\usepackage{tikz-cd}
\usepackage{breqn}
\usepackage{hyperref}
\hypersetup{
    colorlinks=true,
    linkcolor=blue,
    filecolor=magenta,      
    urlcolor=cyan,
    pdftitle={Overleaf Example},
    pdfpagemode=FullScreen,
    citecolor=cyan,
    }
\usepackage{blkarray}
\usepackage{color}
\usepackage{comment}
\usepackage{enumerate}
\usepackage{upgreek}

\usepackage[ruled,vlined]{algorithm2e}
\usepackage{algorithmic}
\usepackage{enumerate}
\usepackage{amsfonts}                  
\usepackage{enumitem}
\usepackage[normalem]{ulem}

\usepackage{subfig}
\usepackage{graphicx}
\usepackage{threeparttable}

\newtheorem{claim}{Claim}

\newtheorem{theorem}{Theorem}
\newtheorem{definition}{Definition}

\newtheorem{obs}{Observation}
\newtheorem{lemma}{Lemma}
\newtheorem{remark}{Remark}
\newtheorem{cor}{Corollary}
\newtheorem{example}{Example}

\newcommand{\M}{{\mathcal{M}}}

\newcommand{\N}{{\mathbb{N}}}

\newcommand{\Ss}{{\mathcal{S}}}
\newcommand{\DR}{{\hbox{DR}}}

\usepackage{blkarray}
\usepackage{multirow}

\usepackage{blkarray, bigstrut}
\usepackage{xparse}

\usepackage{mhchem}

\usepackage[toc,page]{appendix}

\begin{document}

\title{Internal Closedness and von Neumann-Morgenstern Stability \\  in Matching Theory: Structures and Complexity
}

\author{Yuri Faenza\thanks{IEOR Department, Columbia University. Supported by the NSF Award \emph{CAREER: An algorithmic theory of matching markets}. Part of the research leading to this paper was carried out when Yuri Faenza was a visiting scientist in the \emph{Online and Matching-Based Market Design} at Simons Institute in Fall 2019.}, Clifford Stein\thanks{IEOR Department, Columbia University.}, and Jia Wan\thanks{University of Oxford. Part of the research leading to this paper was carried out while Jia Wan was a student at Columbia University.}}
\maketitle
\begin{abstract}

Let $G$ be a graph and suppose we are given, for each $v \in V(G)$, a strict ordering of the neighbors of $v$. A set of matchings ${\cal M}$ of $G$ is called \emph{internally stable} if there are no matchings $M,M' \in {\cal M}$ such that an edge of $M$ blocks $M'$. The sets of \emph{stable} (\`a la Gale and Shapley) matchings and of \emph{von Neumann-Morgenstern stable} matchings are examples of internally stable sets of matching. In this paper, we study, in both the marriage and the roommate case, inclusionwise maximal internally stable sets of matchings. We call those sets \emph{internally closed.} By building on known and newly developed algebraic structures associated to sets of matchings, we investigate the complexity of deciding if a set of matchings is internally closed or von Neumann-Morgenstern stable, and of finding sets with those properties.
\end{abstract}

\section{Introduction}\label{sec:intro}

An instance of the \emph{marriage model} (or a \emph{marriage instance}) is given by a pair $(G,>)$, where $G=(V,E)$ is a bipartite graph and $>$ denotes the set  $\{>_x\}_{x \in V}$, where $>_x$ is a strict ordering of nodes -- also called \emph{agents} -- adjacent to $x$ in $G$. The model was introduced by Gale and Shapley~\cite{gale1962college}, that concurrently also defined the concept of stability. An edge $ab \in E$ is a \emph{blocking pair} for a matching $M$ of $G$ if $b>_a M(a)$ and $a >_b M(b)$ (here and throughout the paper, we let $M(x)$ be the agent $x$ is matched to in matching $M$, and $\emptyset$ if such agent does not exist). A matching is \emph{stable} if it has no blocking pairs. 

Stable matchings enjoy a rich mathematical structure, that has been leveraged on to design various algorithms since Gale and Shapley's publication (see, e.g.,~\cite{Irving,manlove2013algorithmics}). The tractability of stable matchings has boosted the number of applications of the work by Gale and Shapley, that nowadays is used to assign students to schools, doctors to hospitals, workers to firms, partners in online dating, agents in ride-sharing platforms, and more, see, e.g.,~\cite{abdulkadirouglu2003school,hitsch2010matching,roth1984stability,wang2018stable}.   

Although immensely popular, the approach by Gale and Shapley cannot encompass all the different features that arise in applications. Some criticisms targeted the limitation of stability, which may lead to a matching half the size of a maximum-size matching, and more generally may disqualify a beneficial matching because of the preferences of a single pair of agents. For instance, by analyzing real data from the student-public schools market, Abdulkadiro{\u{g}}lu, Pathak, and Roth~\cite{abdulkadirouglu2009strategy} observed that dropping the stability constraint may lead to matchings that are considerably preferred by the student population.

Investigating solution concepts alternative to stability has therefore become an important area of research. However, such alternative concepts often lack many of the attractive structural and algorithmic properties of stable matchings. For instance, the set of Pareto-optimal matchings is more elusive than that of stable matchings, and other than algorithms for finding a few (yet significant) Pareto-optimal matchings, not much is known. Much research has therefore been devoted to defining alternative solution concepts that are tractable, often reducing in a way or another to stable matchings themselves~\cite{abraham2004pareto,abraham2007stable,cseh2022understanding,faenza2021quasi,faenza2021stable,faenza2021affinely,huang2021popularity,manlove2013algorithmics}. 

Another line of criticism to the approach by Gale and Shapley focuses on the marriage model itself, often considered not informative enough of preferences of agents, or not flexible enough to model interactions in certain markets. More general models include those replacing strict preference lists with preference lists with ties~\cite{irving2000hospitals} or choice functions~\cite{roth1984stability}. In this paper, we will be concerned with the classical extension obtained by dropping the bipartiteness requirement on the input graph $G$, leading to the \emph{roommate model}~\cite{gusfield1988structure,irving1985efficient}.

\paragraph{Internal Stability, von Neumann-Morgenstern Stability, and Internal Closedness.} We study two solution concepts alternative to stability, in both the marriage and the roommate model. The definitions of those concepts rely on the classical game theory notion\footnote{Unfortunately, standard terminology overloads the term \emph{stability}. However, one can easily distinguish among those concepts since stability \`a la Gale and Shapley does not have any attribute associated to it.} of \emph{internal stability}~\cite{von2007theory}. For two matchings $M,M'$ of a roommate instance, define the following relation: $M$ \emph{blocks} $M'$ if there is an edge of $M$ that blocks $M'$. A set ${\cal S}$ of matchings is internally stable if there are no $M,M' \in {\cal S}$ such that $M$ blocks $M'$. 
Note that while stability is a property of a matching, internal stability is a property of a set of matchings. Moreover, the family of stable matchings is clearly internally stable. 

As discussed by von Neumann and Morgenstern~\cite{von2007theory}, a family of internally stable solutions ${\cal S}$ to a game (a family of matchings, in our case) can be thought of as the family of ``standard behaviours'' within a social organization: solutions in ${\cal S}$ are all and only those that are deemed acceptable according to current rules. As also argued in~\cite{ehlers2020legal,von2007theory}, one can however question why other solutions are excluded from being members ${\cal S}$. Hence, in order for an internally stable set ${\cal S}$ to be deemed an acceptable standard, it is often required that ${\cal S}$ satisfies further conditions. A typical one requires ${\cal S}$ to be \emph{externally stable}: for every matching $M' \notin {\cal S}$, there is a matching $M \in {\cal S}$ such that $M$ blocks $M'$. A set that is both internally and externally stable is called \emph{von Neumann-Morgenstern} (\emph{vNM}) \emph{stable}~\cite{lucas1992neumann,von2007theory}.

von Neumann and Morgenstern believed vNM stability to be the main solution concept for cooperative games~\cite{ehlers2007neumann,von2007theory}, and the book by Shubik~\cite{shubik1982game} reports more than 100 works that investigate this concept until 1973 (see also~\cite{lucas1992neumann}). However, recent research has shifted the focus from vNM stability to the core of games, which enjoys stronger properties, and is often easier to work with, than vNM stable sets~\cite{ehlers2007neumann}. Interestingly, when we restrict to bipartite matching problems, vNM stability enjoys many of the nice properties that it does not satisfy in more general settings, such as existence, unicity, and algorithmic tractability~\cite{ehlers2007neumann,ehlers2020legal,faenza2018legal,wako2010polynomial}. Conversely, to the best of our knowledge, no prior research has investigated vNM stability in the roommate model.

To define internal closedness, we relax the external stability condition to inclusionwise maximality. We therefore dub \emph{internally closed} an inclusionwise maximal set of internally stable matchings.  Our definition is motivated by two considerations. First, a vNM stable set may not always exist in a roommate instance (see Example~\ref{exp:no_LP_no_SM} in Section~\ref{sec:vNM not exist}), while an internally closed set of matchings always exists. As already observed by von Neumann and Morgenstern~\cite{von2007theory}, existence is a fundamental condition for a solution concept to be considered acceptable\footnote{Indeed, von Neumann and Morgenstern~\cite{von2007theory} seem to suggest that, for the games they consider, vNM stable sets always exist, which was later disproved~\cite{lucas1968game}.}. Second, a central planner may want a specific set of internally stable matchings ${\cal S}$ to be deemed as feasible solutions. It is easy to see that, already in the marriage case, ${\cal S}$ may not be contained in a vNM stable set. Conversely, an internally closed set of matchings containing ${\cal S}$ always exists, as one can start from ${\cal S}$ and iteratively enlarge it as to achieve inclusionwise maximality while preserving internal stability. As we will see, our analysis of internally closed sets of matching will also lead to an algorithmic understanding of vNM stable sets, thus showing the implication of internal closedness for classical stability notions.
\subsection{Overview of Contributions and Techniques} 

Motivated by the discussion above, in this paper we investigate structural and algorithmic properties of vNM stable and internally closed sets of matching. Our results also have implications for the algebraic theory of stable matchings more generally. It is worth presenting our contributions from the end -- that is, from their algorithmic implications. 

\paragraph{Algorithmic and Complexity Results.} We show that, in a marriage instance, we can find in polynomial time an internally closed set of matchings containing any given internally stable set of matchings. Conversely, in a roommate instance, even deciding if a set is internally closed, or vNM stable, is co-NP-hard, and the problem of finding a vNM stable set of matchings is also co-NP-hard. 
All these results are based on structural insights on internally closedness, which in turns imply structural insights for (stable) matching problems more generally. Those are discussed next.

 \paragraph{From Matchings to Edges.} The statement of the algorithmic results in the previous paragraph glossed over the complexity issue of representing the input and the output to our problems, since a family of internally stable matchings may have size exponential in the size of $(G,>)$. This is problematic, since state-of-the-art algorithms for stable matchings and related tractable concepts run in time polynomial in the size of $(G,>)$. We bypass this concern by showing that every internally closed set of matchings coincides with the set ${\cal S}'$ of stable matchings in a \emph{subinstance} $(G',>)$ of the original instance $(G,>)$, i.e., an instance obtained from $G$ by removing certain edges. Hence, the input will be given by a set $E_0\subseteq E(G)$, that will implicitly describe the set ${\cal S}'$ of stable matchings in $(G[E_0],>)$. Similarly, a closed or vNM stable set of matchings will be described compactly by $E_C\subseteq E(G)$. See Section~\ref{sec:prelim} for details.

The observation above allows us to work with (polynomially-sized) sets of edges, rather than with (possibly exponentially-sized) sets of matchings. Our next step lies in understanding how to enlarge the input set of edges $E_0$, as to obtain the set $E_C$. The challenge of this task lies in the fact that adding \emph{any single edge} to $E_0$ may not lead to a strictly larger internally stable set of matchings. Moreover, certain choices of edges added to $E_0$ may prevent the addition of other edges to $E_0$, thus precluding the possibility of finding an internally closed set of matching. Such a phenomenon is illustrated in the example below.

\begin{example}\label{ex:intro}
We consider an instance of the marriage problem with 8 agents $1, 2, 3, 4$ and $A, B, C, D$. The reader can follow the construction in Figure~\ref{fig:ex:first}. Preferences of agents are arranged in table $T$, where for each agent we list their possible partners in order of preference (see Section~\ref{sec:prelim-first} for a discussion on this representation). $T$ has exactly one stable matching $M_1=\{1A, 2B, 3C, 4D\}$. $\{M_1\}$ is then internally stable. If we let  $T^*$ be the table that contains all and only the edges from $M_1$, then $\{M_1\}$ is the set of stable matchings in $T^*$.

We now want to construct an instance containing $T^*$ (and contained in $T$) whose associated set of stable matching is internally closed. The reader can check that the addition of any single edge to $T^*$ leads to an instance whose set of stable matchings is still $\{M_1\}$. On the other hand, adding edges $M_2=\{1B, 2A, 3D, 4C\}$ to $T^*$ gives $T'$, whose family of stable matchings is $\{M_1,M_2\}$. One can check that $\{M_1,M_2\}$ is the unique internally closed family of matchings containing $\{M_1\}$, and it is also the unique vNM stable set in the instance described by $T$. On the other hand, if we had selected the ``wrong'' edge $3A$ to be added to $T^*$, then we would have obtained $T''$. The family of stable matchings associated to $T''$ or any table containing $T''$ is $\{M_1\}$, hence adding $3A$ to $T^*$ prevents us from finding an internally closed (or vNM stable) set of matchings. 

\begin{figure}
\begin{center}
$T=$ \, \begin{tabular}{||c|c||}
\hline 
Agent & Preference List \tabularnewline
\hline 
\hline 
$1$ & $A, B$ \tabularnewline
\hline 
$2$ & $B,A$ \tabularnewline
\hline 
$3$ & $C,A,D$ \tabularnewline
\hline 
$4$ & $D,C$ \tabularnewline
\hline 
$A$ & $3,2,1$ \tabularnewline
\hline 
$B$ & $1,2$ \tabularnewline
\hline 
$C$ & $4,3$ \tabularnewline
\hline 
$D$ & $3,4$ \tabularnewline
\hline 
\end{tabular}
\quad  \quad 
$T^*=$ \, \begin{tabular}{||c|c||}
\hline 
Agent & Preference List \tabularnewline
\hline 
\hline 
$1$ & $A$ \tabularnewline
\hline 
$2$ & $B$ \tabularnewline
\hline 
$3$ & $C$ \tabularnewline
\hline 
$4$ & $D$ \tabularnewline
\hline 
$A$ & $1$ \tabularnewline
\hline 
$B$ & $2$ \tabularnewline
\hline 
$C$ & $3$ \tabularnewline
\hline 
$D$ & $4$ \tabularnewline
\hline 
\end{tabular}
\par\end{center}

\begin{center}
$T'=$ \, \begin{tabular}{||c|c||}
\hline 
Agent & Preference List \tabularnewline
\hline 
\hline 
$1$ & $A, B$ \tabularnewline
\hline 
$2$ & $B,A$ \tabularnewline
\hline 
$3$ & $C,D$ \tabularnewline
\hline 
$4$ & $D,C$ \tabularnewline
\hline 
$A$ & $2,1$ \tabularnewline
\hline 
$B$ & $1,2$ \tabularnewline
\hline 
$C$ & $4,3$ \tabularnewline
\hline 
$D$ & $3,4$ \tabularnewline
\hline 
\end{tabular}
\quad  \quad 
$T''=$ \, \begin{tabular}{||c|c||}
\hline 
Agent & Preference List \tabularnewline
\hline 
\hline 
$1$ & $A$ \tabularnewline
\hline 
$2$ & $B$ \tabularnewline
\hline 
$3$ & $C, A$ \tabularnewline
\hline 
$4$ & $D$ \tabularnewline
\hline 
$A$ & $3, 1$ \tabularnewline
\hline 
$B$ & $2$ \tabularnewline
\hline 
$C$ & $3$ \tabularnewline
\hline 
$D$ & $4$ \tabularnewline
\hline 
\end{tabular}
\par\end{center}
\caption{On the upper left: the instance $T$ of the roommate problem. On the upper right: the table $T^*$ containing all and only the stable edges in $T$. On the bottom left: table $T'$ representing the vNM stable and internally closed set. On the bottom right: $T''$ which adds edge $3A$ to $T^*$.}\label{fig:ex:first}
\end{figure}
\end{example}

Example~\ref{ex:intro} suggests that one needs to iteratively add \emph{sets of edges} rather than single edges, possibly leading to a search space that is exponentially large. However, classical algebraic ideas on stable matchings come to our rescue. 

\smallskip

\paragraph{From Edges to Rotations: the Marriage Case.} Let us first focus on the marriage case, which is investigated in Section~\ref{sec:marriage}. Since vNM stable sets in this model are well-understood~\cite{ehlers2007neumann,faenza2018legal,wako2010polynomial}, we consider only internally closed sets of matchings. A fundamental result attributed in~\cite{knuth1976mariages} to Conway states that stable matchings can be arranged in a distributive lattice under a simple partial order $\succeq_X$ (see Section~\ref{sec:alternative-char}). In turns, this distributive lattice can be described more succinctly as the family of closed sets of an associated poset, whose elements are called \emph{rotations} and whose associated partial order is denoted by $\sqsubseteq$~\cite{Irving}. Rotations are certain cycles in the marriage instance (see Section~\ref{sec:rotations}). 

We use the lattice of matchings and the poset of rotations to prove two characterizations of internally closed sets of matchings. The first shows that a set of matchings is internally closed if and only if it is inclusionwise maximal among the sets of matchings whose lattice wrt (with respect to) $\succeq_X$ shares with the lattice of stable matchings two properties: being closed wrt the join operator, and satisfying the opposition of interest (see Section~\ref{sec:lattice-stable} for definitions). So although we lose stability when moving from the set of stable matchings to an internally closed set of matchings, we keep some of the nice properties that are implied by stability. 

The second characterization shows a certain ``maximality'' property of the poset of rotations associated to an internally closed set of matchings. Roughly speaking, a family of matchings is internally closed if and only if in the poset of rotations associated to it, no rotation can be \emph{dissected}, i.e., replaced with $2$ or more new rotations, and we cannot \emph{vertically expand} the poset by adding rotations that are maximal or minimal wrt the partial order $\sqsubseteq$ (see Section~\ref{sec:alternative-char}). A crucial role in this characterization is played by \emph{generalized rotations}, also introduced in this work (see Section~\ref{sec:generalized-rotations}).

This second characterization of internally closed sets of matchings is then employed by an algorithm that iteratively enlarges the set $E_0$ by trying to dissect rotations or to vertically expand the associated poset. This approach leads to a polynomial-time algorithm for constructing an internally closed set of matchings containing a given set of internally stable set of matchings. See Theorem~\ref{algo:marriage_IStoIC-MC_poly}.

\paragraph{From Edges to Rotations: the Roommate Case.} The roommate case is investigated in Section~\ref{sec:roommate_case} and Section~\ref{sec:vNM_is_NP_hard}. As in the marriage case, the set of stable matchings in the roommate case, when it exists, can be described in terms of a poset of (suitably defined) rotations~\cite{Irving}. However, the resulting structure is less neat. In particular, we are not aware of any interesting lattice structure that can be associated to the family of stable matchings in the roommate case. 
Similarly to the marriage case, our goal is to understand when the poset of rotations associated to a set of stable matchings can be extended, leading to a larger internally stable set of matchings. Our main structural contribution here is to show that the set of stable matchings of an instance is internally closed if and only if the poset of rotations cannot be augmented via what we call a \emph{stitched rotation} (see Definition~\ref{def:stitched_rotation} and Theorem~\ref{augment_dual}). We then show how to construct, from any instance $\phi$ of 3-SAT, a roommate instance that admits a stitched rotation if and only if $\phi$ is solvable, thus proving the claimed hardness result (see Section~\ref{sec:NP_hard_deciding_internally_closed}). 

 

The notion of internally closed sets of matchings has important implications. It provides us with a new tool to understand vNM stability in roommate instances. Most notably, the reduction above can be extended to prove co-NP-hardness of deciding if a set of matchings is vNM stable, and of the problem of finding a vNM stable set. Those results are shown in Section \ref{sec:vNM_is_NP_hard}.

\subsection{Further Discussion on Related Work} To the best of our knowledge, internally closed sets of matchings have not been studied in the literature, and are first defined in this work. Most of the related research has focused on vNM stable sets of matchings in the bipartite case. In their works on legal / vNM stable matchings, Wako~\cite{wako2010polynomial} and Faenza and Zhang~\cite{faenza2018legal} give polynomial-time algorithms to construct, in a bipartite model (respectively, in the one-to-one and in the one-to-many case), the (unique) vNM stable set. In particular, their algorithms constructs \emph{one} internally closed set of matchings -- namely, one containing the set of stable matchings. In contrast, our algorithm allows us to obtain an internally closed set of matchings containing \emph{any} given set of (internally stable) matchings. Interestingly, the algorithms from~\cite{faenza2018legal,wako2010polynomial} obtain a set of matchings whose associated poset of rotations strictly contains the poset of rotations associated to the original set of matchings. This is not the case for our algorithm. In a work predating~\cite{faenza2018legal,wako2010polynomial}, Ehlers~\cite{ehlers2007neumann} proves many strong structural properties of vNM stable sets in the marriage model. Ehlers and Morrill~\cite{ehlers2020legal} investigate two definitions of dominance relations among matchings, which lead to different definitions of internal and external stability. They also restrict their study to bipartite models, although they allow the preferences of one side of the agents to be described by certain choice functions, thus generalizing the strict preference lists from the marriage model. More work in this area on the bipartite case has appeared in~\cite{ehlers2007neumann,herings2017stable,mauleon2011neumann}, including assignment games~\cite{nunez2013neumann}. vNM stable sets do not seem to have been studied in the roommate case, while the related concept of vNM fairsighted stability has been investigated by Klaus et al.~\cite{klaus2011farsighted}. 

\section{Preliminaries}\label{sec:prelim}

\subsection{Roommate Instances, Preference Tables, and Matchings}\label{sec:prelim-first}

\paragraph{Standard Notation.} We denote by $\triangle$ the symmetric difference operator between sets and, for $n \in \mathbb{N}$, we let $[n]=\{1,\dots,n\}$ and $[n]_0=\{0\} \cup [n]$. $\uplus$ denotes the union operator between disjoint sets.

\paragraph{Basic Notions and Preference Table Representation of Roommate Instances.} In this paper, in addition to using the classical graph representation for marriage and roommate instances defined in Section~\ref{sec:intro}, it will be useful to present those instances via preference tables. This representation is not new, see, e.g.,~\cite{Irving}. A (\emph{Roommate}) \emph{instance} is therefore described as follows. Let $A$ be the set of all agents. Each agent $z\in A$ has a \emph{preference list} consisting of a subset $A(z)\subseteq A\backslash \{z\}$ and a strict ordering (i.e., without ties) of elements from this list. For each agent $z\in A$, $A(z)$ contains therefore all and only the agents that are acceptable to $z$. The collection of all preference relations is then represented by the \emph{preference table} $T(A, >)$ (or $T$ in short), where $>$ collects, for each agent $z\in A$, the strict ordering within the preference list of agent $z$, denoted as $>_z$. We assume that  preference tables are symmetric, i.e., $z_1$ is on $z_0$'s preference list if and only if $z_0$ is on $z_1$'s preference list. See Figure~\ref{fig:ex:first}, top left, for an example of a preference table $T$ with $8$ agents. 

For $z, z_1,z_2\in A$ with $z_1, z_2\in A(z)$, we say that $z$ \emph{strictly prefers} $z_1$ to $z_2$ if $z_1 >_{z} z_2$; and we say that $z$ \emph{(weakly) prefers} $z_1$ to $z_2$ if $z_1 >_{z} z_2$ or $z_1=z_2$, and write $z_1 \geq_{z} z_2$. For $z\in A$, we additionally assume that $z_1 >_{z} \emptyset$ for all $z_1\in A(z)$, that is, all agents strictly prefer being matched to some agent in their preference list than being left unmatched. 

Because of the symmetry assumption, we say that $z_0z_1\in T$ if $z_1$ is on $z_0$'s preference list (equivalently, if $z_0$ is on $z_1$'s preference list) and call $z_0z_1$ an \emph{edge} \emph{(of $T$)}. The $\emph{deletion}$ of edge $z_0z_1$ from $T$ means to delete $z_0$ from $z_1$'s preference list and delete $z_1$ from $z_0$'s preference list. Let $r_T(z_0,z_1)$ denote the \emph{ranking} (i.e., the position, counting from left to right) of $z_1$ within the preference list of $z_0$ in preference table $T$. Therefore, $r_T(z_0,z_1)<r_T(z_0,z_2)$ if and only if agent $z_0$ strictly prefers $z_1$ to $z_2$ within preference table $T$. We let $r_T(z, \emptyset)=+\infty$.

For a given preference table $T$ and agent $z$, let $f_{T}(z)$, $s_{T}(z)$, $\ell_{T}(z)$ denote the first, second and last agent on $z$'s preference list in $T$. 


\paragraph{Consistency, Subtables.} We say that two rommmate instance $T', T''$ have \emph{consistent} preference lists if, for any pair of edges $z_0z_1$, $z_0z_2$ such that $z_0z_1, z_0z_2\in T$ and $z_0z_1, z_0z_2\in T'$, we have $r_T(z_0,z_1)<r_{T}(z_0,z_2)$ if and only if $r_{T'}(z_0,z_1)<r_{T'}(z_0,z_2)$. We write $T' \subseteq T$ when $T',T$ are instances with the following properties: (a) $z_0z_1 \in T$ for all $z_0z_1 \in T'$ and (b) $T,T'$ have consistent preference lists. In this case, we call $T'$ a \emph{subtable} of $T$. 

Let $T_1, T_2 \subseteq T$ be two subtables of a preference table $T$, and note that they must have consistent preference lists by definition. Let $T_1\cap T_2$ be the subtable of $T$ that contains all edges that are in both $T_1$ and $T_2$; let $T_1\cup T_2$ be the subtable of $T$ that contains all edges that are in $T_1$, $T_2$, or both.

\paragraph{Matching Basics.}
Fix a roommate instance $T(A,>)$. A \emph{matching} $M$ of $T$ is a collection of disjoint pairs of agents from $A$, with the property that $M\subseteq T$. For $z_0 \in A$, we let  $M(z_0)$ be the partner of $z_0$ in matching $M$. If $z_0z_1 \notin M$ for every $z_1 \in A$, we write $M(z_0)=\emptyset$. If $M(z_0)\neq \emptyset$, we say that $z_0$ is \emph{matched} (\emph{in $M$}). If, for matchings $M,M'$ and an agent $z$ we have $M(z)\geq_z {M'}(z)$, we say that $z$ \emph{weakly prefers} $M$ to $M'$; if $M(z)>_z {M'}(z)$, then we say $z$ \emph{strongly prefers} $M$ to $M'$. 

For a matching $M\subseteq T$, we say that $ab \in T$ is a \emph{blocking pair} for $M$ if $b >_a M(a)$ and $a >_b M(b)$. We also say that $M$ is \emph{blocked} by $ab$. A matching $M$ is \emph{stable} if it is not blocked by any pair, and \emph{unstable} if it is blocked by some pair. We use the notion of blocking pair to define the following binary relation among matchings: matching $M'$ \emph{blocks} matching $M$ if $M'$ contains a blocking pair $ab$ for $M$. Note that any matching $M$ can be interpreted as a preference table, where, for each $z_0z_1\in M$, $z_1$ is the only agent appearing in $z_0$'s preference list. We can therefore write $M\subseteq T$, $T\cup M$, and $T\cap M$ using the preference table interpretation of matchings. 

For a roommate instance $T$, we let $\mathcal{M}(T)$ denote the set of matchings of $T$, and $\mathcal{S}(T)$ denote the set of stable matchings of $T$. 
If $xy \in T$ is contained in some stable matching of $T$, then it is called a \emph{stable edge} or \emph{stable pair}. The subtable $E_S(T)$ obtained from $T$ by removing all unstable edges is called the \emph{stable subtable} of $T$. If $E_S(T)=T$, then $T$ is called a \emph{stable table}. When $T$ is clear from the context, we abbreviate ${\cal M}(T), {\cal S}(T), r_T(\cdot, \cdot),  E_S(T), \dots$ by ${\cal M}$, ${\cal S}$, $r(\cdot, \cdot), E_S, \dots$. 

The following result extends structural properties of matching that do not block each other (see, e.g.,~\cite{faenza2018legal}) to the roommate case. Given two matchings $M,M'$, edge $ab  \in M\cup M'$ is called \emph{irregular} (\emph{for $M, M'$}) if both $a$ and $b$ strictly prefer $M$ to $M'$ or both strictly prefer $M'$ to $M$. See Appendix \ref{appendix_sym_diff} for a proof.

\begin{lemma}\label{symmetric_diff} Let $T$ be a roommate instance, and $M,M'\in \mathcal{M}$. Assume that $M$ does not block $M'$ and $M'$ does not block $M$. Then:
\begin{enumerate}
    \item $G[M\triangle M']$ is a disjoint union of singletons and even cycles. Hence, a node is matched in $M$ if and only if it is matched in $M'$. 
    \item No edge from $M\cup M'$ is irregular for $M,M'$.
\end{enumerate}
\end{lemma}

\subsection{Internal Stability and Related Concepts}
\paragraph{Internal Stability.} We now define new concepts tailored for this paper, and related basic facts. Let $T$ be a roommate instance. We say that a set of matchings $\M'\subseteq \M(T)$ is \emph{internally stable} if given any two matchings $M_0, M_1\in \M'$, $M_0$ does not block $M_1$ and $M_1$ does not block $M_0$. 
\begin{obs}\label{obs:subset}
Let $T$ be an instance of the roommate problem, ${\cal M}''\subseteq {\cal M}'\subseteq {\cal M}(T)$ and ${\cal M}'$ be internally stable. Then ${\cal M}''$ is internally stable.
\end{obs}

For an internally stable ${\M}' \subseteq \M(T)$, we define its closure $\overline{\M'}=\{M \in \M(T) : \{M\} \cup \M' \hbox{ is internally stable}\}$. Note that $\overline{{\M}'}$ may not be internally stable. If $\overline{\M'}=\M'$, we say that $\M'$ is \emph{internally closed}. 
By Observation~\ref{obs:subset}, internally closed sets of matchings are exactly the inclusionwise maximal internally stable sets. For internally stable sets $\M'' \supseteq {\cal M}'$, we say that $\M''$ is an \emph{internal closure} of ${\cal M}'$ if $\M''$ is internally closed. Clearly, every internally stable set of matching admits an internal closure. Note also that an internal closure of a set of internally stable matchings is, by definition, internally closed. Moreover, the internal closure of an arbitrary internally stable set of matchings is not necessarily unique.

The following lemmas give basic structural results of internally stable and internally closed sets of matchings. 

\begin{lemma}\label{lem:internally-stable-are-stable}
Let $T$ be an instance of the roommate problem, $\M'\subseteq \M(T)$, and $\widetilde{T}=\cup\{M|M\in \mathcal{M}'\}$. $\M'$ is internally stable if and only if $\mathcal{M}'\subseteq \Ss(\widetilde{T})$.
\end{lemma}
\begin{proof}
$(\Rightarrow)$ Let $M \in {\cal M}'$. Clearly $M \subseteq \widetilde T$. Let $e \in \widetilde T$. By construction, there exists $M' \in {\cal M}'$ such that $e \in M'$. By internal stability, $M'$ does not block $M$. Hence, $e$ does not block $M$. We conclude that $M \in {\cal S}(\widetilde T)$, hence ${\cal M}'\subseteq {\cal S}(\widetilde T)$. 

$(\Leftarrow)$ The set of stable matchings of an instance $\widetilde{T}$ is internally stable. By Observation~\ref{obs:subset}, so is every subset of it. 
\end{proof}

When we restrict to internally closed sets of matchings, one direction of Lemma~\ref{lem:internally-stable-are-stable} can be strengthened.

\begin{lemma}\label{lem:from-internally-closed-to-edges}
Let $T$ be an instance of the roommate problem, $\M'\subseteq \M(T)$ and $\widetilde{T}=\cup\{M|M\in \M'\}$. If $\M'$ is internally closed, then $\M'=\Ss(\widetilde{T})$.
\end{lemma}
\begin{proof}

${\cal M}'\subseteq {\cal S}(\widetilde T)$ follows by Lemma~\ref{lem:internally-stable-are-stable} and by the fact than internally closed sets are internally stable. Conversely, we claim that a matching $M \in {\cal S}(\widetilde{T})$ also belongs to ${\M'}$. Indeed, $M$ is not blocked by any matching from $\M' \subseteq \Ss(\widetilde{T})$ by definition of stability. Moreover, $M$ does not block any matching from ${\M'}\subseteq {\cal S}(\widetilde{T})$, again by definition of stability. So $M\in \overline{\M'}=\M'$, where the equality follows by definition of internally closed. 
\end{proof}

Note that the converse of Lemma~\ref{lem:from-internally-closed-to-edges} is, in general, not true: there are stable subtables $\widetilde T$ of $T$ such that ${\cal S}(\widetilde T)$ is not internally closed. Because of Lemma~\ref{lem:from-internally-closed-to-edges}, we can however succinctly represent any internally closed set of matchings $\M'$ via a subtable $\widetilde{T}$ of $T$  such that $\M'={\cal S}(\widetilde{T})$. Note that the function from subtables to internally closed sets of matching becomes one to one if we restrict to subtables $\cup\{M|M\in \mathcal{M}'\}$, for ${\cal M}'$ looping over all internally closed sets of matchings. Observe moreover that the representation of internally closed sets of matchings as sets of stable matchings  implies that they inherit structural and algorithmic properties of the latter. This mapping also implies that the number of internally closed sets of matchings is upper bounded by a singly exponential function in the number of the agents.

\paragraph{vNM Stability.} A set of matchings ${\M}'\subseteq {\M}(T)$ is called \emph{externally stable} (in $T$) if for each $M \in \M(T)\setminus {\cal M}'$, there exists $M' \in {\cal M}'$ that blocks $M$. A set that is both internally and externally stable is called \emph{vNM stable} (in $T$).  

\begin{lemma}\label{lem:vNM=IC+ES}
A set of matchings $\M'\subseteq \M(T)$ is vNM stable if and only if $\M'$ is internally closed and $\M'$ is externally stable.
\end{lemma}
\begin{proof}
$(\Leftarrow)$ If $\M'$ is internally closed, then by definition $\M'$ is internally stable. Combined with the fact that $\M'$ is externally stable, $\M'$ is vNM stable.

$(\Rightarrow)$ By hypothesis, $\M'$ is internally and externally stable. Suppose by contradiction that there exists some matching $\widetilde{M}\in \M\setminus \M' $ such that $\widetilde{M}\cup \M'$ is internally stable. Hence, $\widetilde{M}$ is not blocked by any matching in $\M'$, contradicting external stability. Therefore $\M'$ is internally closed. 
\end{proof}

Lemma~\ref{lem:from-internally-closed-to-edges} and Lemma~\ref{lem:vNM=IC+ES} imply that for each vNM stable set $\M'$, we have $\M'=\Ss(\widetilde T)$, with $\widetilde T=\cup \{ M : M \in \M'\}$.

\subsection{The problems of interest}\label{sec:coNP}

In this paper, we study the following four problems. 
\smallskip

\noindent %
\noindent\fbox{\parbox[c]{1\textwidth - 2\fboxsep - 2\fboxrule}{%
\begin{center}
\textbf{Find an internally closed set containing a given internally stable set, Marriage Case \texttt{(IStoIC-MC)}}
\par\end{center}
\textbf{Given:} A marriage instance $T$ and a stable table $\widetilde{T}$ such that $\widetilde{T}\subseteq T$. \\
\textbf{Find:} $T'\subseteq T$ 
such that $\Ss(T')$ is an internal closure of $\Ss(\widetilde{T})$.
}}

\smallskip

\noindent %
\noindent\fbox{\parbox[c]{1\textwidth - 2\fboxsep - 2\fboxrule}{%
\begin{center}
\textbf{Check Internal Closedness($\widetilde{T}, T$) \texttt{(CIC)} }
\par\end{center}
\textbf{Given:} A roommate instance $T$ and a stable table $\widetilde{T}$ such that ${\widetilde{T}}\subseteq T$.\\
\textbf{Decide:} If ${\cal S}(\widetilde{T})$ is internally closed.  
}}

\smallskip 

\noindent %
\noindent\fbox{\parbox[c]{1\textwidth - 2\fboxsep - 2\fboxrule}{%
\begin{center}
\textbf{Check vNM Stability($\widetilde{T}, T$) \texttt{(CvNMS)}}
\par\end{center}
\textbf{Given:} A roommate instance $T$ and a stable table $\widetilde{T}$ such that ${\widetilde{T}}\subseteq T$.\\
\textbf{Decide:} If $\Ss(\widetilde{T})$ is vNM stable. %
}}

\smallskip 

\noindent %
\noindent\fbox{\parbox[c]{1\textwidth - 2\fboxsep - 2\fboxrule}{%
\begin{center}
\textbf{Find a vNM Stable Set($T$) \texttt{(FvNMS)}}
\par\end{center}
\textbf{Given:} A solvable roommate instance $T$.\\
\textbf{Find:} $T'\subseteq T$ such that $\Ss(T')$ is a vNM stable set, or conclude that no vNM stable set exists. %
}}

\smallskip

As discussed in the introduction, for complexity reasons, the input to the first three problems above contains a table $\widetilde{T}\subseteq T$, which implicitly describes the associated set of matchings ${\cal S}(\widetilde{T})$. Moreover, $\widetilde{T}$ is assumed to be stable. Those two hypothesis are motivated by Lemma~\ref{lem:from-internally-closed-to-edges}, which implies that each internally closed set is of the form ${\cal S}(T')$ for some stable table $T'$ with $T'\subseteq T$. 

Moreover, by definition, the internal closure of a internally stable set of matchings always exists, hence \texttt{IStoIC-MC} is a search problem.

\smallskip

\section{Internally closed sets: the marriage case}\label{sec:marriage}

A \emph{marriage} instance~\cite{gale1962college} is a special case of a roommate instance, satisfying the additional constraint that $A$ can be partitioned into two disjoint sets $X=\{x_1, x_2, \dots, x_p\}$ and $Y=\{y_1, y_2, \dots, y_q\}$, where the preference list of any agent in $X$ consists only of a subsets of agents in $Y$ (hence, by symmetry, the preference list of any agent in $Y$ consists only of a subset of agents in $X$). An agent $x \in X$ is called an \emph{$X$-agent}, and similarly $y \in Y$ is called a \emph{$Y$-agent}. We denote a marriage instance by $T(X\cup Y, >)$ or succinctly by $T$. All the other notations extend from the roommate setting. Here is an example of a preference table $T$ for a bipartite instance, with agents $X=\{x_1,x_2\}$, $Y=\{y_1,y_2\}$: 
\begin{center}
\begin{tabular}{||c|c||}
\hline 
Agents & Preference List \tabularnewline
\hline 
\hline 
$x_1$ & $y_1, y_2$ \tabularnewline
\hline 
$x_2$ & $y_2,y_1$ \tabularnewline
\hline 
$y_1$ & $x_2,x_1$ \tabularnewline
\hline 
$y_2$ & $x_2,x_1$ \tabularnewline
\hline 
\end{tabular}
\par\end{center}

The goal of this section is twofold: first, we give two alternative characterizations of internally closed sets of matchings in the marriage case (see Theorem \ref{thm:internally_closed_characterization}) -- one that is purely algebraic, and another one that relies on a generalization of the classical concept of rotations, which is also introduced in the current section. Then, we use the second characterization to show the following. 
\begin{theorem}\label{algo:marriage_IStoIC-MC_poly}
\texttt{IStoIC-MC} can be solved in time $O(n^4)$, where $n$ is the number of agents in the marriage instance in input.
\end{theorem}

\subsection{About the lattice of stable matchings}\label{sec:lattice-stable}

We start by discussing known features of the poset of stable matchings obtained via an associated partial order. For an extensive treatment of this topic, see~\cite{Irving}. Throughout this section, we fix a marriage instance $T$. Define the following domination relationship between matchings: $$M\succeq_X M' \hbox{ if, for every $x\in X$, $M(x)\geq_x M'(x)$}.$$ That is, all agents in $X$ weakly prefer $M'$ to $M$. If, in addition, there exists at least one agent $x\in X$ such that $M(x)>_x M'(x)$ (or, equivalently, if $M'\neq M$), then we write $M\succ_X M'$. We symmetrically define the relations $M\succeq_Y M'$ and $M\succ_Y M'$.

\begin{theorem}\label{thm:X-optimal-matching} Let $Z \in \{X,Y\}$. $({\cal S}(T),\succeq_Z)$ forms a distributive lattice. Its greatest element is denoted by $M^T_Z$ (or simply $M_Z$) and called the \emph{$Z$-optimal} stable matching. $({\cal S}(T),\succeq_Z)$ satisfies the following two properties, defined next: it is closed under the join and meet operators, and it satisfies the opposition of interests. $({\cal S}(T),\succeq_X)$ is called the \emph{lattice of stable matchings} \end{theorem} 

\paragraph{The join and meet operators.}
Given two matchings $M, M'$, we let their \emph{join} be $$M \lor M'=\{xy : x \in X, y \in \{M(x), {M'}(x)\} : y \geq_x M(x), y \geq_x {M'}(x)\},$$ i.e., every agent in $X$ is paired with their better partner among $M$ and $M'$, and their \emph{meet} be $$M \land M'=\{xy : y \in Y, x \in \{M(y), {M'}(y)\}: x \geq_y M(y), x \geq_y {M'}(y)\},$$ i.e., every agent in $Y$ is paired with their better partner among $M$ and $M'$. Note that, in general, the join (resp., meet) of two matchings is not a matching. 
A set of matchings ${\cal M}'\subseteq \M(T)$ is \emph{closed} wrt the join (resp., meet) operator if, for every $M,M' \in \M'$, we have that $M\lor M' \in {\cal M}'$ (resp., $M\land M' \in {\cal M}')$. 

\paragraph{Opposition of interests.}
A set of matchings ${\cal M}'\subseteq \M(T)$ satisfies the \emph{opposition of interests} if for any $M,M'\in \M'$, $M\succeq_X M'$ if and only if $M'\succeq_Y M$.


\subsection{Classical and generalized rotations}\label{sec:rotations}

\subsubsection{Classical rotations and properties}
Fix again a marriage instance $T$. We sum up here definitions and known facts about the classical concept of rotations. See again~\cite{Irving} for extended discussions and proofs.

\begin{definition}\label{def:classical_rotations}
Let $M\in \Ss(T)$. Given distinct $X$-agents $x_0, \dots, x_{r-1}\in X$ and $Y$-agents $y_0, \dots, y_{r-1}\in Y$, we call the finite sequence of ordered pairs \begin{equation}\label{eq:rotation-marriage}\rho= (x_0, y_0), (x_1, y_1), \dots, (x_{r-1}, y_{r-1})\end{equation} a (\emph{classical}) \emph{X-rotation exposed in $M$} if, for every $i \in [r-1]_0$:
\begin{enumerate}[label=\alph*)]
    \item $x_iy_i\in M$;
    \item $x_i>_{y_{i+1}} x_{i+1}$;
    \item $y_{i} >_{x_i} y_{i+1}$;
    \item $M(y) >_y x_i$ for all $y \in Y$ such that $x_iy \in T$ and $y_i >_{x_i} y >_{x_i} y_{i+1}$.
\end{enumerate}
where all indices are taken modulo $r$. We often omit ``$X-$'' and call $\rho$ simply a rotation.
\end{definition}

Note that conditions a)-b)-c) only depends on edges $x_i y_i$ and $x_{i}y_{i+1}$. We call them therefore \emph{basic conditions}. Condition d) depend on other edges from the table $T$, and we therefore call it \emph{$T$-dependent condition.}


Note that an $X$-rotation (and a generalized $X$-rotation, to be introduced in Section~\ref{sec:generalized-rotations}) with $r$ elements is equivalent up to a constant shift (modulo $r$) of all indices of its pairs. Hence, we will always assume that indices in entries of a (generalized) $X$-rotation are taken modulo $r$. 

\begin{figure}[htbp]\centering
\def\sym#1{\ifmmode^{#1}\else\(^{#1}\)\fi}
\begin{minipage}[c]{.35\textwidth}
\begin{center}
$T=$ \, \begin{tabular}{||c|c||}
\hline 
Agent  & Preference list \tabularnewline
\hline 
$x_1$  & $y_4,y_1,y_3$ \tabularnewline
\hline 
$y_1$  & $x_2,x_1,x_4$ \tabularnewline
\hline 
$x_2$ & $y_2,y_1$\tabularnewline
\hline 
$y_2$ & $x_3,x_2$\tabularnewline
\hline 
$x_3$ & $y_3,y_2$\tabularnewline
\hline 
$y_3$ & $x_1,x_3$\tabularnewline
\hline 
$x_4$ & $y_1,y_4$\tabularnewline
\hline 
$y_4$ & $x_4,x_1$\tabularnewline
\hline 
\end{tabular}
\end{center}
\end{minipage}\begin{minipage}[c]{.3\textwidth}
\centering
\begin{tikzpicture}[node distance={15mm}, thick, main/.style = {draw, circle}]
\node[main] (1) {$x_1$}; 
\node[main] (2) [left of=1] {$y_4$}; 
\node[main] (3) [below right of=1] {$y_3$};
\node[main] (4) [below left of=2] {$x_4$};
\node[main] (5) [below of=4] {$y_1$};
\node[main] (6) [below of=3] {$x_3$};
\node[main] (7) [below left of=6] {$y_2$};
\node[main] (8) [left of=7] {$x_2$};
\draw[->] (2) -- (1); 
\draw[->] (4) -- (2); 
\draw[->] (1) -- (3);
\draw[->] (1) -- (5);
\draw[->] (5) -- (4);
\draw[->] (3) -- (6);
\draw[->] (6) -- (7);
\draw[->] (7) -- (8);
\draw[->] (8) -- (5);
\end{tikzpicture}
\end{minipage}\begin{minipage}[c]{.35\textwidth}
\begin{center}$T'=$ \, \begin{tabular}{||c|c||}
\hline 
Agent  & Preference list \tabularnewline
\hline 
$x_1$  & $y_4,y_3$ \tabularnewline
\hline 
$y_1$  & $x_2,x_4$ \tabularnewline
\hline 
$x_2$ & $y_2,y_1$\tabularnewline
\hline 
$y_2$ & $x_3,x_2$\tabularnewline
\hline 
$x_3$ & $y_3,y_2$\tabularnewline
\hline 
$y_3$ & $x_1,x_3$\tabularnewline
\hline 
$x_4$ & $y_1,y_4$\tabularnewline
\hline 
$y_4$ & $x_4,x_1$\tabularnewline
\hline 
\end{tabular}
\end{center}
\end{minipage}
\caption{Illustration from Example~\ref{exp:rotations}, Example~\ref{exp:gen_rotations}, and Example~\ref{exp:gen_rotations-2}. On the left: Table $T$. In the center: the graph $D_X(M)$. On the right: Table $T'$.}
\label{fig:exp:gen_rotations}
\end{figure}
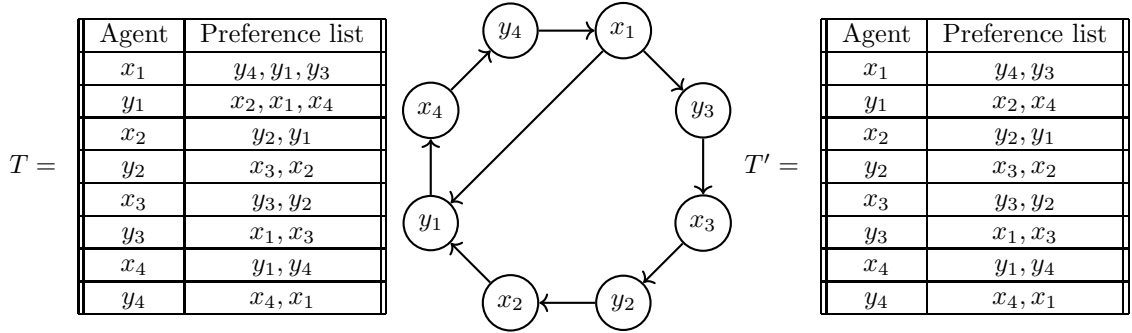

\begin{example}\label{exp:rotations}
Consider the marriage instance $T$ in Figure \ref{fig:exp:gen_rotations}, left, and its stable matching $M=\{x_1y_4,x_2y_2,x_3y_3,x_4y_1\}$. One can easily check that $\rho=(x_1, y_4), (x_4, y_1)$ satisfies properties a)-b)-c)-d) from Definition \ref{def:classical_rotations}, hence it is an $X$-rotation exposed in $M$.
\end{example}

We abuse notation and write $x \in \rho$ (resp.~$y \in \rho$) if $(x,y) \in \rho$ for some $y$ (resp., for some $x$). The \emph{elimination} of an $X$-rotation $\rho$ exposed in a stable matching $M$ maps $M$ to the matching $M'\coloneqq M/\rho$ where $M(x)={M'}(x)$ for $x\in X\setminus \rho$ and ${M'}(x_i)=y_{i+1}$ for all $i \in [r-1]_0$. Intuitively, the matching $M'=M/\rho$ differs from $M$ by a cyclic shift of each $X$-agent in $\rho$ to the partner in $M$ of the next X-agent in $\rho$. Moreover, $M$ \emph{immediately precedes $M'$} in the lattice of stable matchings of $T$. That is, $M\succ_X M'$ and there is no $M'' \in {\cal S}(T)$ such that $M\succ_X M'' \succ_X M'$. 

\begin{theorem}\label{thm:rotations}
There exists exactly one set of classical $X$-rotations $R_X=\{\rho_1,\rho_2,\dots,\rho_q\}$ such that, for $i \in [q]$, $\rho_i$ is exposed in $(M/\rho_1)/\rho_2)\dots \rho_{i-1}$, and  
$$M_Y=M_X/R_X=((M_X / \rho_1) / \rho_2) / \dots ) / \rho_q.$$  Moreover, $R_X$ is exactly the set of all X-rotations exposed in some stable matching of $T$.
\end{theorem}


Extending the definition from the previous theorem, for $R=\{\rho_1,\rho_2,\dots,\rho_k\}\subseteq R_X$, such that, for $i \in [k]$, $\rho_i$ is exposed in $(M/\rho_1)/\rho_2)\dots \rho_{i-1}$, we let $$M/R\coloneqq((M / \rho_1) / \rho_2) / \dots ) / \rho_k.$$

Define the poset $(R_X,\sqsubseteq)$ as follows: $\rho'\sqsubseteq \rho$ if for any sequence of $X$-rotation eliminations $M_X/\rho_0/ \rho_1/\dots/\rho_k$ with $\rho=\rho_k$, we have $\rho' \in \{\rho_0,\rho_1,\dots,\rho_k\}$. If moreover $\rho'\neq \rho$, we write $\rho'\sqsubset \rho$. 

A set $R\subseteq R_X$ is called \emph{closed} if $$\rho \in R, \rho' \in R_X(T) : \rho' \sqsubseteq \rho \, \Rightarrow \, \rho' \in R.$$
For $\rho \in R_X$, we let $R(\rho)=\{ \rho' \in R_X: \rho' \sqsubset \rho\}$. Note that $R(\rho)$ is closed and does not include $\rho$.

The following extension of Theorem~\ref{thm:rotations} can be seen as a specialization of Birkhoff's representation theorem~\cite{birkhoff1937rings} to the lattice of stable matchings. 
\begin{theorem}\label{thm:closed-rotation}
The following map defines a bijection between closed sets of rotations and stable matchings:
$$ R \subseteq R_X, \ R \hbox{ closed } \ \rightarrow \ M/R.$$
\end{theorem}

When we want to stress the instance $T$ we use to build the poset of X-rotation, we denote the set of all X-rotations by $R_X(T)$. 



\paragraph{Partition of the stable subtable.} Recall that the stable subtable $E_S(T)$ of $T$ (defined in Section~\ref{sec:prelim-first}) contains all and only the edges that are in some matching $M \in {\cal S}(T)$. $E_S(T)$ can be partitioned as follows.

\begin{theorem}\label{thm:decomposition-stable-subtable}
$E_S(T)=M_Y \uplus \{xy|(x,y) \in \rho \text{ for some } \rho\in R_X(T)\}$.
\end{theorem} 

\subsubsection{Generalized $X$-rotations and properties}\label{sec:generalized-rotations}  

Fix again a marriage instance $T$. We now introduce an extension of the classical concept of $X$-rotation, which we call \emph{generalized $X$-rotation}, and define its associated digraph. 

\begin{definition}\label{def:generalized_rotations}
Let $M\in \M(T)$. Given distinct $X$-agents $x_0, \dots, x_{r-1}$ and $Y$-agents $y_0, \dots, y_{r-1}$, we call the finite sequence of ordered pairs \begin{equation}\label{eq:generalized-rotation}\rho^g= (x_0, y_0), (x_1, y_1), \dots, (x_{r-1}, y_{r-1})\end{equation} a \emph{generalized X-rotation exposed in $M$} if, for every $i \in [r-1]_0$, it satisfies properties a)-b)-c) from Definition~\ref{def:classical_rotations} (but not necessarily d).  We often omit ``$X-$'' and call $\rho^g$ simply a generalized rotation.
\end{definition}

To distinguish from classical $X$-rotations, we use the superscript $g$ for all generalized $X$-rotations throughout the paper. Note that a classical $X$-rotation exposed in a matching $M$ is also a generalized rotation exposed in $M$. As we did for classical $X$-rotations, we again write $x \in \rho^g$ (resp.~$y \in \rho^g$) if $(x,y) \in \rho^g$ for some $y$ (resp., for some $x$). The \emph{elimination} of a generalized $X$-rotation, or of a set of generalized $X$-rotations, is defined analogously to the classical rotation case. among the generalized rotation $\rho^g$ exposed at a matching $M$, we still have $M \succ_X M/\rho^g $.

For a generalized $X$-rotation $\rho^g$ as in~\eqref{eq:generalized-rotation}, we let  $E(\rho^g)=\{x_iy_i\}_{i \in [r-1]_0}\cup \{x_iy_{i+1}\}_{i \in [r-1]_0}$, and often interpret $E(\rho^g)$ as a subtable of $T$. Hence, $T\cup E(\rho^g)$ is a well-defined table. We also write $T \cup \rho^g$, interpreting $\rho^g$ as a subtable with edges $\{x_iy_i\}_{i \in [r-1]_0}$. We extend this notation to a set $R=\{\rho^g_1,\dots,\rho^g_k\}$ of generalized rotations, letting $T \cup R = T \cup_{j=1}^k \rho_i^g$ and $T\cup E(R) = T \cup_{i=1}^k E(\rho_i^g)$.

\paragraph{The generalized rotation digraphs.} We define the following generalized X-rotation digraph for a matching $M$ of $T$, denoted as $D_X(M,T)$ or simply $D_X(M)$ when $T$ is clear from the context. The set of nodes is given by $X\cup Y$.
For any agents $x \in X, y \in Y$, add arc $(x,y)$ if $x >_{y} M(y)$ and and $M(x) >_x y$; add arc and $(y,x)$ if $M(y)=x$. Note that the outdegree of each $X$-agent can be larger than $1$, but the outdegree of every $Y$-agent is at most $1$. The next lemma follows directly from the  definition of $D_X(M)$ and it is similar to a known statement for classical rotations, see, e.g.,~\cite{Irving}.

\begin{lemma}\label{lem:generalized-rotation-digraph} Let $M$ be a matching. $x_0\rightarrow y_1\rightarrow x_1\rightarrow\dots\rightarrow x_{r-1}\rightarrow y_{0}\rightarrow x_0$ is a cycle in $D_X(M)$ if and only if $\rho^g = (x_0, y_0), (x_1, y_1),\dots, (x_{r-1}, y_{r-1})$ is a generalized $X$-rotation exposed in $M$. We say that $\rho^g$ and the cycle \emph{correspond} to each other.\end{lemma}

Classical and generalized $Y$-rotations are defined similarly to classical and generalized $X$-rotations, with the role of agents in $X$ and $Y$ swapped. By symmetry, all definitions and properties carry over.

\begin{example}\label{exp:gen_rotations}
Consider again the instance $T$ from Example \ref{exp:rotations} and define matching $$M=\{x_1y_4,x_2y_2,x_3y_3,x_4y_1\}.$$ The generalized rotation digraph $D_X(M,T)$ is given in Figure~\ref{fig:exp:gen_rotations}, center.

$P=x_1\rightarrow y_3\rightarrow x_3 \rightarrow y_2 \rightarrow x_2 \rightarrow y_1 \rightarrow x_4 \rightarrow y_4 \rightarrow x_1$ is a cycle in $D_X(M)$. Cycle $P$ corresponds to a generalized X-rotation $\rho^g=(x_1, y_4), (x_3,y_3), (x_2,y_2), (x_4,y_1)$, which is exposed in $M$ (wrt table $T$). One can check that $\rho^g$ satisfies properties a)-b)-c) but not d) from Definition \ref{def:classical_rotations}. \end{example}


\subsection{Two characterizations of internally closed sets of matchings}\label{sec:alternative-char}

We next state two alternative characterizations of internally closed sets of matchings. The first characterization pertains the lattice of matchings: we show that a set of matchings is internally closed set if and only if it is inclusionwise maximal among the sets of matchings that share with $(\Ss(T),\succ_X)$ the two properties, closedness wrt to the join operator, and opposition of interests. The second characterization is based on rotations: we show that a set of matchings is internally closed if and only if the corresponding poset of rotations cannot be ``augmented'' in a well-defined manner. We prove the first characterization in Subsection \ref{sec:IC_algebraic} and the rotation-based characterization in Subsection \ref{sec:IC_algorithmic}.

\begin{theorem}\label{thm:internally_closed_characterization}
Let $T$ be a marriage instance and ${\cal M}'\subseteq {\cal M}(T)$. The following are equivalent.
\begin{enumerate}
    \item ${\cal M}'$ is internally closed.
    \item (Algebraic characterization) ${\cal M}'$ is inclusionwise maximal among the subsets of $\M(T)$ that: \begin{enumerate}[label=\alph*)]
        \item are closed wrt the the join operator;
        \item satisfy the opposition of interests.
    \end{enumerate}
    \item (Generalized rotations characterization) ${\cal M}'={\cal S}(T')$ for a stable subtable $T' \subseteq T$ such that the following holds:
       \begin{enumerate}[label=\alph*)]\item $D_X(M^{T'}_Y,T)$ and $D_Y(M^{T'}_X,T)$ have no cycle. 
    \item \texttt{DR}$(T,T',\rho)$ returns $\{\rho\}$ for all $ \rho \in R_X(T')$.
\end{enumerate}
\end{enumerate}
\end{theorem}

Before we move to the proof of Theorem~\ref{thm:internally_closed_characterization}, let us make two observations. First, part a) of the algebraic characterization could alternatively be defined in terms of closedness with respect to the meet operator. In particular, an internally closed set is closed wrt both the join and the meet operator. Second, the generalized rotations characterization relies on the solution of the \texttt{DR}$(\cdot)$ problem, which is defined in Section~\ref{sec:IC_algorithmic}.


\subsubsection{Algebraic Characterization of Internally Closed Sets}\label{sec:IC_algebraic}

\begin{lemma}\label{lem:IC_algebraic}
Let $T$ be a marriage instance and ${\cal M}'\subseteq {\cal M}(T)$. ${\cal M}'$ is an internally closed set of matchings if and only if it is an inclusionwise maximal subset of ${\cal M}(T)$ among those that are a) closed wrt the join operator and b) satisfy the opposition of interests.
\end{lemma}
\begin{proof}
$(\Leftarrow)$ We first show that for any $M, M'\in \M'$, $M$ and $M'$ do not block each other. We will not use inclusionwise maximality of $\M'$, so the statement applies to any set of matchings satisfying a) and b). 

Suppose by contradiction that $x_0y_0$ blocks $M'$ for some $x_0y_0 \in M$. Hence, $y_0 >_{x_0} {M'}(x_0)$ and $x_0 >_{y_0} {M'}(y_0)$, so $x_0y_0\in (M\lor M') \setminus M'$. Now consider the dominance relation between $M\lor M'$ and $M'$:
by definition, $M\lor M' \succeq_X M'$. Since $x_0y_0\in (M\lor M')\setminus M'$, we have $M\lor M' \succ_X M'$. Because $\M'$ is closed wrt the join operator, $M\lor M' \in \M'$. Since $\M'$ satisfies the opposition of interests, we must have $M' \succ_Y M\lor M'$. However, ${(M\lor M')}(y_0)=x_0>_{y_0} {M'}(y_0)$. We reached a contradiction, so $M$ and $M'$ do not block each other.

Therefore $\M'$ is internally stable. Now let $\M''\subseteq \M(T)$ be an internally stable set such that $\M'\subseteq \M^{''}$. By Lemma~\ref{lem:internally-stable-are-stable}, $\M''\subseteq \Ss(\widetilde{T})$ where $\widetilde{T}=\cup\{M:M\in \M''\}$. By Theorem~\ref{thm:X-optimal-matching}, $\Ss(\widetilde{T})$ satisfies a) and b). Since $\M'$ is an inclusionwise maximal set satisfying a) and b) and we have $\M' \subseteq \M^{''} \subseteq \Ss(\widetilde{T})$, the three inclusions are equalities. In particular, $\M'=\M''$. Hence, $\M'$ is an inclusionwise maximal internally stable set -- that is, it is internally closed.


$(\Rightarrow)$ Let $\M'$ be an internally closed set. By Lemma \ref{lem:from-internally-closed-to-edges}, $\M'=\Ss(\widetilde{T})$ for some $\widetilde{T} \subseteq T$. By Theorem~\ref{thm:X-optimal-matching}, $\Ss(\widetilde{T})$ satisfies a) and b). Suppose for a contradiction that there exists a family of matchings $\M''\subseteq  \M\setminus \M'$ such that $\M'' \cup \M'$ satisfies a) and b). Then because of the first part of the proof, $\M''\cup \M'$ is internally stable. This contradicts $\M'$ being internally closed, concluding the proof.  
\end{proof}

\subsubsection{Generalized Rotation Characterization of Internally Closed Sets}\label{sec:IC_algorithmic}

\paragraph{Dissection of a rotation.} Key to the proof of the second characterization on internally closed sets are the problem of dissecting a rotation and the concept of dissecting set. We introduce them next. An illustration is given in Example~\ref{exp:gen_rotations-2}.

\smallskip

\noindent %
\noindent\fbox{\parbox[c]{1\textwidth - 2\fboxsep - 2\fboxrule}{%
\begin{center}
\textbf{Dissecting a Rotation (\texttt{DR})}
\par\end{center}
\textbf{Given:} Marriage instances $T'\subseteq T$ with $T'$ stable, 
and $\rho\in R_X(T')$. \\
\textbf{Find:} A set $R=\{\rho_1, \rho_2,\dots,\rho_k\}$ satisfying the following properties: a) $R\subseteq R_X(T^*)\setminus R_X(T')$ with $T^*=T' \cup_{j=1}^k \rho_j$; b) $M^{T'}_X/R(\rho)/\rho= M^{T'}_X /R(\rho) / R$; or output $\{\rho\}$ if $R$ as above does not exist.}}

\smallskip

If $\DR(T,T',\rho)$ outputs a set $R \neq \{\rho\}$, then $R$ is called a \emph{dissecting set} for $(T,T',\rho)$. Note that a dissecting set has at least two elements. 

\smallskip

\begin{example}\label{exp:gen_rotations-2} Consider again the instance $T$ we studied in Example~\ref{exp:rotations} and Example~\ref{exp:gen_rotations}, and recall that we defined $\rho=\rho^g=(x_1, y_4), (x_3,y_3), (x_2,y_2), (x_4,y_1)$. Let $T'=T \cup \{\rho\}$ (see Figure~\ref{fig:exp:gen_rotations}).
$T$ has $3$ inclusionwise maximal matchings:
$$M=\{x_1y_4,x_2y_2,x_3y_3,x_4y_1\}, \quad M_1=\{x_1y_1,x_2y_2,x_3y_3,x_4y_4\}, \quad 
    M_2=\{x_1y_3,x_2y_1,x_3y_2,x_4y_4\}.$$
$M$ is the X-optimal matching within $T'$ and $\rho$ is the only (classical) $X$-rotation exposed in $M$ in $T'$. Note that we have $M_2=M/\rho$. Consider now table $T^*=T=T' \cup \rho_1 \cup \rho_2$, where 
$$\rho_1=(x_1, y_4), (x_4, y_1), \quad \rho_2=(x_1,y_1),(x_3,y_3), (x_2,y_2).$$
$M$ is also the X-optimal matching within $T$. Note that $\rho_1 \in R_X(T^*)$ is exposed in $M$, $\rho_2 \in R_X(T^*)$ is exposed in $M/\rho_1$, and $M_2=M/\rho_1/\rho_2$. Yet, $\rho_1,\rho_2\not\in R_X(T')$. Hence, $\{\rho_1,\rho_2\}$ is a dissecting set for $(T,T',\rho)$. Note that when going from table $T'$ to table $T^*=T$, the set of stable matchings becomes strictly larger, i.e. $\Ss(T)\supsetneq \Ss(T')$. We are able to achieve this result by dissecting rotation $\rho$ into two rotations $\rho_1$ and $\rho_2$, enlarging the poset of rotations from a singleton to a chain with two elements. 
\end{example}

\paragraph{Dissecting a rotation preserves stability and enlarges the rotation poset.}

Intuitively, the \texttt{DR} procedure replaces a single rotation $\rho$ in $T'$ with a set of rotations $\DR(T, T', \rho)=R$ such that: $R$ are rotations in a table $T^*$ with $T'\subseteq T^*\subseteq T$, but not in $T'$, yet iteratively eliminating $R$ in $T^*$ starting from a matching $M$ where $\rho$ is exposed leads to the same matching obtained by eliminating $\rho$ in $T'$ starting from $M$. Hence, the poset of rotations in $T^*$ is an ``expansion'' of the poset of rotations in $T'$, since $\rho$ has been replaced by a set of at least $2$ rotations.

\begin{lemma}\label{lem:dissecting-preseeves-stability}
Let $T$ be a marriage instance, and $T'\subseteq T$ with $T'$ stable. Let $\rho \in R_X(T')$ and $R=\{\rho_1,\dots,\rho_{k}\}$ be a dissecting set for $(T,T',\rho)$. Then a) $T^*=T' \cup_{j=1}^k \rho_j$ is a stable table and b) $R_X(T^*)=R_X(T')\setminus \{\rho\} \cup \{\rho_1,\dots, \rho_k\}$.
\end{lemma}

\begin{proof}
Let $M^0 = M_X^{T'}/ R(\rho)$, and $M^*=M^0/\rho$. By construction, $E(R)\subseteq M^* \cup_{j=1}^k \rho_j$ and $\rho_j \in R_X(T^*)$ for $j \in [k]$. 

We first claim that $\Ss(T')\subseteq \Ss(T^*)$. To show it, it suffices to prove that no $xy \in T^*\setminus T'$ blocks $M \in \Ss(T')$. Suppose $y>_x M(x)$. By construction, $M^0(x)>_x y$ and $M^*(y)>_y x$. We apply a stronger version of the opposition of interests property (see, e.g.,~\cite[Theorem 1.3.1]{Irving}), that states that, if $M', M''$ are stable matchings with $x'y' \in M'\setminus M''$, then exactly one of the following holds: $M'(x')>_{x'}M''(x')$, or $M'(y')>_{y'}M''(y')$. Since $M^0(x)>_x y >_x M(x)$, we deduce that $M(y)>_y M^0(y)$. On the other hand, since $M^*$ immediately follows $M^0$ in $(\Ss(T'),\succeq_X)$, we have $M(y)\geq_y M^*(y)>_y x$, proving the claim.

Since $T'$ is a stable table, all edges from $T'$ are therefore stable in $T^*$. By Theorem~\ref{thm:decomposition-stable-subtable}, for $i \in [k]$, all edges from $\rho_i$ are stable in $T^*$. Hence, $T^*$ is a stable table, concluding the proof of a).

In order to show $b)$, let $\rho' \in R_X(T')\setminus \{\rho\}$. By construction, it is a generalized rotation exposed in $M^* \in \Ss(T^*)$ wrt $T^*$. To conclude $\rho' \in R_X(T^*)$, we show that the $T^*$-dependent condition holds. Consider a sequence of eliminations of all rotations from $R_X(T')$, whose existence is guaranteed by Theorem~\ref{thm:rotations}. Let $M$ be the matching from which $\rho'$ is eliminated and $M'=M/\rho'$. We argued above that ${\cal S}(T')\subseteq {\cal S}(T^*)$, hence $M \in {\cal S}(T^*)$. Assume first that $\rho'$ precedes $\rho$ in this ordering. Then each $X$-agent $x \in \rho$ satisfies $M'(x) \geq_x M^0(x)$. Moreover, each $X$-agent $x \notin \rho$ satisfies $M^0(x)=M^*(x)$, hence it has the same preference list in $T'$ and $T^*$. Hence, $M(x)\geq_x M'(x) \geq_x M^0(x)>_xy$ for all $xy \in T^*\setminus T'$. Since the $T'$-dependent condition holds by the assumption $\rho' \in R_X(T')$, also the $T^*$-dependent condition holds. Similarly, using $M'(y)\geq_y M(y) \geq_y M^*(y)$ for every $Y$-agent  $y\in \rho$ sets the case when $\rho$ follows $\rho'$. Thus, $R_X(T^*)\supseteq \{\rho_1,\dots,\rho_k\}\cup R_X(T') \setminus \{\rho\}=:R'$. On the other hand, notice that $M_X/R'=M_Y$, hence b) follows by Theorem~\ref{thm:rotations}. 
\end{proof}


\paragraph{Generalized rotations and internally stable sets.} We next show an intermediate lemma connecting generalized rotations with the operation of enlarging an internally stable set of matchings.

\begin{lemma}\label{mid_dissect_rotation}
Let $T$ be a marriage instance and $T'\subseteq T$ with $T'$ stable. Let $\rho \in R_X({T'})$ be exposed in a matching $M^0\in \Ss(T')$. Let $M^*=M^0/\rho$ and 
$$\widetilde{T}=\{xy\in T| x \in X, y \in Y,  M^0(x) \geq _x y 
\geq_x M^*(x) \text{ and }  M^*(y) \geq_y x \geq_y M^0(y)\}.$$ 

Then:
\begin{enumerate}
    \item Let $\rho^g\neq \rho$ be a generalized rotation exposed in $M^0$ within $\widetilde{T}$, $M^1=M^0/\rho^g$, and  $\overline{T}=T' \cup E(\rho^g)$. Then a) $M^1\not\in\Ss(T')$, b) $\{M^1\}\cup \Ss(T')$ is internally stable,  c) $M^0\succ_XM^1\succ_X M^*$, d) $M^0 \in {\cal S}(\overline T)$, 
    and e) $\rho^g \in R_X({\overline T})$ is exposed in $M^0$.
    \item Conversely, let $M^1 \subseteq T$ be a matching such that $M^0\succ_XM^1\succ_X M^*$ and $\{M^1\}\cup \Ss(T')$ is internally stable. Then there exists a generalized rotation $\rho^g\neq \rho$ exposed in $M^0$  within $\widetilde T$. 
\end{enumerate}
\end{lemma}

\begin{proof} 1.

\begin{claim}\label{cl:bip:1}$\{M^1,M^0,M^*\}$ is internally stable and $M^0\succ_X M^1\succ_X M^*$.
\end{claim}

\noindent \underline{Proof of Claim}. By construction of $\widetilde{T}$, it must hold that $M^*(y) \geq_y M^1(y) \geq_y M^0(y)$ for all $y \in Y$ and $M^0(x) \geq_x M^1(x) \geq_x M^*(x)$ for all $x \in X$. Therefore $\{M^1,M^0,M^*\}$ is internally stable. Moreover, $M^1 \neq M^0, M^*$ since $\rho \neq \rho^g$. Hence, $M^0\succ_X M^1\succ_X M^*$. $\hfill \diamond$

\begin{claim}\label{cl:bip:2} $M^1\not\in \Ss(T')$.
\end{claim}

\noindent \underline{Proof of Claim}. Directly from Claim~\ref{cl:bip:1} and from the fact that $M^0$ immediately precedes $M^*$ in the lattice $({\cal S}(T'),\succeq_X)$. $\hfill \diamond$

\begin{claim}\label{cl:bip:3}${M^1} \cup {\cal S}(T')$ is internally stable.  
\end{claim}

\noindent \underline{Proof of Claim}. Let $M \in {\cal S}(T')\setminus \{M^0, M^*\}$. We show that $M$ does not block $M^1$, and $M^1$ does not block $M$. We start with the former: let $xy\in M$. We can assume $xy \notin M^0 \cup M^*$ because by Claim~\ref{cl:bip:1} $M^0, M^*$ do not block $M^1$. 
Suppose $y>_x M^1(x)$. Applying to $M,M^*$ the stronger version of the opposition of interests property used in the proof of Lemma~\ref{lem:dissecting-preseeves-stability} and the fact that  $y>_x M^1(x)\geq_x M^*(x)$, we deduce $M^*(y)>_y x$. By stability of $T'$ and the fact that $M^0$ immediately precedes $M^*$ in the lattice of stable matchings, we deduce that 
$M^0(y)\geq_y x$. Since ${M^1}(y)\geq_y M^0(y)$ (see Claim~\ref{cl:bip:1}), $xy$ does not block $M^1$. 

Next, suppose $xy \in M^1$ blocks $M$. Since $\{M^0, M^*, M\} \subseteq {\cal S}(T')$, we can suppose $xy \notin M^0 \cup M^*$. Assume $y >_x M(x)$. Applying again the stronger version of the opposition of interests property and an argument similar to the above, we deduce $M(y)>_y x$. This shows that $M$ and $M^1$ do not block each other. Hence, $M^1 \cup {\cal S}(T')$ is internally stable.  $\hfill \diamond$

\smallskip

\begin{claim}\label{cl:bip:4}
$M^0 \in {\cal S}(\overline{T})$.
\end{claim}

\noindent \underline{Proof of Claim}. 
By construction, $\overline{T}=T' \cup M^1$. By hypothesis, $M^0 \in {\cal S}(T')$. Since, by Claim~\ref{cl:bip:1}, $M^0$ is not blocked by $M^1$, we deduce $M^0 \in {\cal S}(\overline T)$. $\hfill \diamond$

\begin{claim}\label{cl:bip:5}
 $\rho^g \in R_X({\overline T})$ is exposed in $M^0$.\end{claim}

\noindent \underline{Proof of Claim}. 
 Let $\rho^g$ be of the form~\eqref{eq:generalized-rotation}. Every $xy \in E(\rho^g)$ either belongs to $M^1$, or to $M^0$. Since $M^1,M^0\subseteq \overline T$, $\rho^g$ is a generalized rotation  exposed in $M^0$ within $\overline T$. We next show the $\overline T$-dependent condition. Since $T'$ is a stable table and $M^0$ immediately precedes $M^*$ in $({\cal S}(T'),\succeq_X)$, there exists no edge $xy\in T'$ such that ${M^0}(x) >_x y >_x {M^*}(x)$. Since $M^0(x)\geq_x M^1(x)\geq_x M^*(x)$, there is  no edge $xy \in T'$ such that $M^0(x)>_x y >_x M^1(x)$. 
 
By construction, an $X$-agent $x \notin \rho^g$ has the same preference list in $\overline{T}$ and $T'$, while an $X$-agent $x \in \rho$ has exactly one more agent $M^1(x)=y_{i+1}$  in $\overline{T}$ with respect to $T'$, with $M^0(x)>_x M^1(x)>_x M^*(x)$. In particular, for each $i \in [r-1]_0$, $y_{i+1}=M^1(x_i)$ is the first agent $y$ in $x_i$'s preference list in $\overline{T}$ satisfying $y_i >_{x_i} y$ and $x_i >_{y} M(y)$.  Hence, the $\overline T$-dependent condition is satisfied.  $\hfill \diamond$


\smallskip 
2. 

\begin{claim}\label{cl:bip:6}
$M^1 \subseteq \widetilde T$.
\end{claim}

\noindent \underline{Proof of Claim}. 
By hypothesis, $M^0\succ_X M^1 \succ_X M^*$. Since $M \cup \Ss(T')$ is internally stable, $M \cup \Ss(T')\subseteq \Ss(T'\cup \{M\})$ by Lemma~\ref{lem:internally-stable-are-stable}. Hence we can apply the opposition of interests properties, and deduce $M^*\succ_Y M^1 \succ_Y M^0$. Hence, $M^1 \subseteq \widetilde T$. $\hfill \diamond$

\smallskip 



Consider the graph $G[M^1\Delta M^0]$. Since $\{M^1\}\cup \Ss(T')$ is internally stable by hypothesis, $M^1$ and $M^0$ do not block each other. We can therefore apply Lemma~\ref{symmetric_diff} and conclude that there exists an even cycle $P$ of $G[M^1\Delta M^0]$ without irregular edges. Let $P=x_0\rightarrow y_1\rightarrow x_1\rightarrow\dots \rightarrow y_{k-1} \rightarrow x_{k-1}\rightarrow y_{0}\rightarrow x_0$, where $x_iy_i\in M^0$ and $x_iy_{i+1}\in M^1$ for all $i \in [k-1]_0$. By the assumption $M^0\succ_XM^1$ and the absence of irregular edges, it must hold that $y_i >_{x_i} y_{i+1}$ and $x_{i} >_{y_{i+1}} x_{i+1}$ for all $i \in [k-1]_0$. Hence, $\rho^g=(x_0,y_0),(x_1,y_1),\dots,(x_{k-1},y_{k-1})$ is a generalized rotation exposed in $M^0$ within table $M^1 \cup M^0\subseteq \widetilde T$, where the containment follows from the definition of $\widetilde T$ and Claim~\ref{cl:bip:6}. Moreover, clearly $M^0 / \rho^g \succeq_X M^1$. Since $M^1 \succ_X M^* = M^0/\rho$, we deduce $\rho^g \neq \rho$, concluding the proof of 2.
\end{proof}

\paragraph{Algorithm for Dissecting a Rotation.} In Algorithm~\ref{alg:DR}, we present the routine \emph{Dissect$\_ $Rotation} that solves \texttt{DR}$(T,T',\rho)$. The algorithm starts from matchings $M^0=M_{X}^{T'}/R(\rho)$ and $M^*=M^0/\rho$ and finds a generalized rotation whose edges are contained in the subtable $\widetilde T$ of $T$ defined as in Lemma~\ref{mid_dissect_rotation}. If the only such generalized rotation is $\rho$, then no dissection is possible and $\rho$ is returned; else, the algorithm returns a dissection set for $(T,T',\rho)$.

\begin{algorithm}[H]
\caption{Dissect\_Rotation$(T,T', \rho)$}
\DontPrintSemicolon
\KwIn{Marriage instance $T$, subtable $T'\subseteq T$ with $T'$ stable, $\rho\in R_X(T')$}
\KwOut{A dissecting set for $(T,T',\rho)$, if it exists; else, $\{\rho\}$.} 
\begin{algorithmic}[1]
\STATE Let $M^0=M_X^{T'} / R(\rho)$ and $M^* = M^0 / \rho$.
\STATE Construct subtable $\widetilde{T}$ as follows: $\widetilde{T}=\{xy\in T| x \in X, y \in Y,  M^0(x) \geq _x y 
\geq_x M^*(x) \text{ and }  M^*(y) \geq_y x \geq_y M^0(y)\}.$ 

\IF{$D_X(M^0, \widetilde{T})$ contains exactly 1 cycle}\label{DR:st:if}
    \STATE return $\{\rho\}$.
\ELSE
    \STATE Find a cycle $C_1$ in $D_X(M^0, \widetilde{T})$ corresponding to a generalized rotation $\rho^g_1\neq \rho$. 
    \STATE Define $M^1=M^0/ \rho^g_1$  and suppose $G[M^1 \triangle M^*]$ contains $k-1$ cycles $C_2,\dots,C_{k-1}$ .
\FOR{$j=2,\dots,k-1$}
    \STATE{Let $C_j=x_0\rightarrow y_1\rightarrow x_1\rightarrow y_2 \rightarrow \dots \rightarrow y_{r-1}\rightarrow x_{r-1}\rightarrow y_0\rightarrow x_0$, where, for $i \in [r-1]_0$, $x_iy_i\in M^1$ and $x_iy_{i+1}\in M^*$.}
    \STATE{Set $\rho^g_j=(x_0,y_0),(x_1,y_1),\dots,(x_{r-1},y_{r-1})$.} 
    \ENDFOR
    \STATE return $\{\rho_1^g,\rho^g_2,\dots,\rho^g_k\}$.
\ENDIF

\end{algorithmic}\label{alg:DR}
\end{algorithm}


\begin{theorem}\label{thm:alg:DR:correct} Algorithm~\ref{alg:DR} correctly solves~\texttt{DR} in time $O(n^2)$, where $n$ is the number of agents in the marriage instance $T$.
\end{theorem}
\begin{proof} 

\begin{claim}\label{cl:bip:9}
If the algorithm outputs $\{\rho\}$, then there exists no dissecting set for $(T,T',\rho)$. \end{claim}

\noindent \underline{Proof of Claim}. 
We start by observing that $D_X(M^0,\widetilde{T})$ has at least one cycle. In fact, by definition, $\rho$ is a (classical) rotation exposed in $M^0$ within subtable $T'$. Since  $\widetilde{T}\supseteq E(\rho)$, we have that $\rho$ is a generalized rotation exposed in $M^0$ within table $\widetilde{T}$. By Lemma~\ref{lem:generalized-rotation-digraph}, generalized rotations exposed in $M^0$ in $\widetilde{T}$ are in bijection with cycles of $D_X(M^0,\widetilde{T})$, and the observation follows. 

Now suppose that there exists a dissecting set $\{\rho_1,\rho_2,\dots,\rho_k\}$ for $(T,T',\rho)$. We show that $D_X(M^0,\widetilde{T})$ contains at least two cycles. Since the algorithm returns $\{\rho\}$ if and only if $D_X(M^0,\widetilde{T})$ has exactly one cycle, the claim then follows. By definition of dissecting set, $M^*=M^0/\rho=(((M^0/\rho_1)/\rho_2)/\dots)/ \rho_k$. Since $k \geq$ 2 by definition of dissecting set, $\rho_1\neq \rho$. By definition, $\rho_1\in R_X({{T^*}})\setminus R_X(T')$, where ${T^*}=T' \cup_{i=1}^k \rho_i$. Since $E(\rho_1) \subseteq \widetilde{T}$ by construction, we have that $\rho_1$ is a generalized rotation exposed in $M^0$ within $\widetilde T$. By Lemma~\ref{lem:generalized-rotation-digraph}, $D_X(M^0,\widetilde{T})$ contains at least two cycles (one corresponding to $\rho$, one corresponding to $\rho^g$), concluding the proof. \hfill $\diamond$

\smallskip 

Because of Claim~\ref{cl:bip:9}, in order to conclude the proof of correctness of the algorithm, we show that, if $D_X(M^0,\widetilde{T})$ contains more than one cycle, then $\{\rho_1^g,\dots,\rho_k^g\}$ as constructed by the algorithm forms a dissecting set for $(T,T',\rho)$. Suppose therefore that there exist two or more cycles in $D_X(M^0,\widetilde{T})$. Then, by Lemma~\ref{lem:generalized-rotation-digraph}, one of those correspond to the generalized rotation $\rho^g_1\neq \rho$ exposed in $M^0$. Define $T^*=T' \cup_{j=1}^k \rho_j$. 

\begin{claim}\label{cl:bip:7}
Let $x \in X$. The preference lists of $x$ in $T^*$ has at most one more entry than its preference list in $T'$. If this new element exists, it coincides with $M^1(x)$, and satisfies $M^0(x)>_xM^1(x)>_xM^*(x)$.
\end{claim}

\noindent \underline{Proof of Claim}. 
Note that $M^0 \succ_X M^1 \succ_X M^*$ by Lemma~\ref{mid_dissect_rotation}, part 1. $C_2,\dots,C_K$ are vertex-disjoint cycles by construction. Using the map defined by Lemma~\ref{lem:generalized-rotation-digraph}, we deduce that the corresponding sets $E(\rho^g_2),\dots, E(\rho^g_k)$ are vertex-disjoint. Since $M^0 \triangle M^* = E(\rho^g_1) \cup E(\rho^g_2) \cup \dots \cup E(\rho_k)$, the thesis follows. $\hfill \diamond$

\begin{claim}\label{cl:bip:8}
$M^0  \in {\cal S}(T^*)$. \end{claim}

\noindent \underline{Proof of Claim}. By hypothesis, $M^0 \in {\cal S}(T')$ and $M^0 \subseteq T^*$. Since $T'$ is a stable table, no edge from $T'$ blocks $M^0$. $T^*\setminus T'$ consists only of edges from $M^1$. By definition of (generalized) rotation, $M^0 \succ_X M^1$. Hence, no edge from $T^*$ blocks $M^0$, concluding the proof. $\hfill \diamond$

\begin{claim}\label{cl:bip:10}
Let $j \in [k]$. Then a) $\rho_{j}^g \in R_{X}(T^*)$, b) $\rho_{j}^g$ is exposed in $M^{j-1}$ within $T^*$, c) $M^{j}  \in {\cal S}(T^*)$,  and d) $M^{j-1}\succ_X M^j$, where $M^j= M^0 / \rho_1^g / \rho_2^g / \dots / \rho_{j}^g$.
\end{claim}

\noindent \underline{Proof of Claim}. The proof is by induction on $j$. Let $j=1$. Define $\rho_1^g$ as in~\eqref{eq:generalized-rotation}, so that for $i \in [r-1]_0$, $x_iy_i \in M^0$ and $x_i y_{i+1} \in M^1$. By Lemma~\ref{mid_dissect_rotation}, part 1, $\rho^g_1 \in R_X(\overline T)$, where $\overline T = T' \cup E(\rho^g_1)$, and $M^0 \succ_X M^1 \succ_X M^*$. In particular, $\rho_1^g$ satisfies the basic conditions for belonging to $R_X(T^*)$. To show that $\rho_1^g$ satisfies the $T^*$-dependent condition, 
observe that by Claim~\ref{cl:bip:7}, for any $i \in [r-1]_0$, $y_i$ and $y_{i+1}$ are consecutive in the preference list of $x_i$ in $T^*$. Thus a) and b) hold. 
Since $M^0 \in {\cal S}(T^*)$ by Claim~\ref{cl:bip:8}, eliminating $\rho_1^g$ from $M^0$ we obtain $M^1$ with $M^0 \succ_X M^1$ and $M^1 \in {\cal S}(T^*)$, showing c) and d).

Now let $j \geq 2$ and again let $\rho^g_j$ be as in~\eqref{eq:generalized-rotation}. Since $C_2,\dots, C_k$ are vertex-disjoint and, by Lemma~\ref{lem:generalized-rotation-digraph}, $\rho_j^g$ is a generalized rotation exposed in $M^1$, we deduce that $\rho_j^g$ is a generalized rotation exposed in $M^{j-1}$  within $T^*$, proving $b)$. By the fact that $M^j=M^{j-1}/\rho_j^g \prec_X M^{j-1}$, we deduce d). It therefore suffices to show a) which together with $M^{j-1} \in {\cal S}(T^*)$ (by induction hypothesis) immediately implies c), concluding the proof. 

To prove a), it suffices to show that, for $i \in [r-1]_0$, there is no $y \in Y$ such that $x_iy \in T^*$ and $y_i >_{x_i} y >_{x_i} y_{i+1}$. As in the case $j=1$, this immediately follows by Claim~\ref{cl:bip:7}. \hfill $\diamond$

\smallskip

To conclude the proof of the correctness of Algorithm~\ref{alg:DR}, we observe that, for $j \in [k]$, $\rho^g_j \notin R_X(T')$. Indeed, for each $j \in [k]$, it is easy to see that $E(\rho^g_j) \not\subseteq T'$. 

The overall running time is $O(n^2)$ since within this time bound we can: a)
construct $D_X(\widetilde{T})$ and  $M^0$~\cite{Irving}; b)
construct $\widetilde{T}$ and the generalized rotation digraph $D_X(M^0,\widetilde{T})$, decide if the latter has at most one cycle, and compute $\rho_1^g$ if it has more~\cite{johnson1975finding}; c)
compute $\{\rho_2^g,\dots,\rho_k^g\}$ as the symmetric difference of two sets of edges.
\end{proof}

\paragraph{Generalized Rotations Exposed in the $Y$-Optimal Stable Matching.} 

Our next lemma shows the effect of adding, to a stable table, edges from a generalized $X$-rotation exposed in the $Y$-optimal stable matching. A symmetric lemma holds switching the roles of $X$ and $Y$. 

\begin{lemma}\label{lem:gen_rotation_IS}
Let $T$ be a a marriage instance and $T'\subseteq T$, with $T'$ stable. Let $M_Y=M_Y^{T'}$ and suppose $\rho^g$ is a generalized $X$-rotation exposed in matching $M_Y$ within table $T$. Let $M^*=M_Y/\rho^g$ and $\overline T=T' \cup E(\rho^g) =T' \cup M^*$. Then a) $M^*\notin {\cal S}(T')$, b) the set $\{M^*\} \cup \Ss(T')$ is internally stable, c) $\overline T$ is a stable table, d) $M^*$ is the Y-optimal stable matching in $\Ss(\overline T)$ and e) $R_X(\overline T)=R_X(T')\cup \{\rho^g\}$. 
\end{lemma}

\begin{proof}
Write $\rho^g$ as in~\eqref{eq:generalized-rotation}. By definition of rotation elimination, ${M^*}(y_{i+1})=x_{i} >_{y_{i+1}} x_{i+1}={M}_Y(y_{i+1})$ for all $i \in [r-1]_0$. Therefore, $M^*\succ_Y M_Y$. By definition, $M_Y\succeq_Y M$ for all $M\in \Ss(T')$. It follows that a) holds and no matching in $\Ss(T')$ can block $M^*$. 

On the other hand, by construction $M^*(x_{i})=y_{i+1} <_{x_i} y_i={M_Y}(x_{i})$ for all $i \in [r-1]_0$. Hence, $M^*(x)\preceq_X M_Y \prec_X M$ for all $M \in {\cal S}(T')$, where the second relation holds by the opposition of interest property and the $Y$-optimality of $M_Y$. Therefore $M^*\prec_X M$ for all $M \in \Ss(T')$. 
It follows that $M^*$ can block no matching in $\Ss(T')$. Therefore b) and, since $T'$ is a stable table, c) hold. 

$M^*$ is Y-optimal in $\Ss(\overline T)$ since for every $y\in Y$, ${M^*}(y)=f_{T'}(y)$, that is, all agents in $Y$ are matched to their first preference within table $\overline T$. Moreover, it is easy to see that $R_X^{\overline T}\supseteq R_X^{T'}$, and if $R_X^{T'}=\{\rho_1,\dots, \rho_k\}$ with $M_Y=\dots(((M/ \rho_1)/\rho_2)/\dots)/\rho_k$, then $M^*=(\dots(((M/ \rho_1)/\rho_2)/\dots)/\rho_k)/\rho^g$. The claim the follows from Theorem~\ref{thm:rotations}.
\end{proof}

\paragraph{Correctness of the characterization.} We now conclude the proof of Theorem~\ref{thm:internally_closed_characterization} by showing that the generalized rotations characterization of internally closed sets is correct.

$(\Rightarrow)$ Suppose ${\cal M}'$ is internally closed. By Lemma~\ref{lem:from-internally-closed-to-edges}, we know that ${\cal M}'=\Ss(T')$ for some stable table $T'\subseteq T$. To show that condition a) holds, suppose by contradiction that, wlog, there exists some cycle in the generalized X-rotation digraph on the $Y$-optimal matching $D_X(M_Y,T')$. By Lemma \ref{lem:gen_rotation_IS}, there exists $M^*\in \M(T)\setminus \Ss(T')$ such that $\{M^*\}\cup \Ss(T')$ is internally stable, contradicting $\Ss(T')$ being internally closed. Condition b) follows from Lemma \ref{lem:dissecting-preseeves-stability}.

$(\Leftarrow)$ Suppose ${\cal M'}$ is not internally closed. If there exists no stable table $T'\subseteq T$ such that ${\cal M}'={\cal S}(T')$, we are done by Lemma~\ref{lem:from-internally-closed-to-edges}. We therefore suppose such a $T'$ exists. By hypothesis, there exists $M \in {\cal M}(T)\setminus \Ss(T')$ such that $\Ss(T')\cup \{M\}$ is internally stable. By hypothesis, $\Ss(T')\cup \{M\} \subseteq {\cal S}(T' \cup M)$. We first show that we can pick $M$ so that one of the following holds: (i) $M\succ_X M^T_X=:M_X$, (ii) $M\prec_X M^T_Y=:M_Y$, or (iii) $M_X\succ_X M \succ_X M_Y$. Suppose that for any $M \in \overline{\Ss(T')} \setminus \Ss(T')$, none of these holds. Pick $M \in \overline{\Ss(T')} \setminus \Ss(T')$ and assume that $M$ is not comparable to $M_X$ wrt $\succ_X$ (the case $M$ being not comparable to $M_Y$ following analogously). Among all $M \in \overline{\Ss(T')} \setminus \Ss(T')$, pick $M$ that is maximal wrt $\succ_X$. Note that for such $M$ we still have that $M$ is incomparable to $M_X$. We can then apply Theorem~\ref{thm:X-optimal-matching} and deduce $M'=M \lor M_X  \in \Ss(T'\cup M)$. Since $\Ss(T')\cup \{M'\}\subseteq \Ss(T'\cup M)$, we deduce that $\Ss(T')\cup \{M'\}$ is internally stable by Lemma~\ref{lem:internally-stable-are-stable}. However, $M' \succ_X M$ since $M_X$ and $M$ are incomparable wrt $\succ_X$, contradicting the choice of $M$.
 

We next show that if one of (i)-(ii)-(iii) holds, one of a),b) from the statement of the characterization is not satisfied. 

Suppose that (i) holds. By Lemma \ref{symmetric_diff}, $G[M\Delta M_X]$ consist only of even cycles whose all edges are regular. Take any such cycle $C=y_0\rightarrow x_1\rightarrow y_1\rightarrow\dots \rightarrow x_{r-1}\rightarrow y_{r-1}\rightarrow x_0\rightarrow y_0$, where for $i\in [r-1]_0$, $y_ix_i \in M_X$ and $y_{i}x_{i+1} \in M$. Since $M\succ_X M_X$, we have $y_i>_{x_{i+1}} y_{i+1}$ and, by regularity of the edges of $C$, $x_i>_{y_i} x_{i-1}$. Hence, $(y_0,x_0),(y_1,x_1),\dots,(y_{r-1},x_{r-1})$ is a generalized Y-rotation exposed in $M_X$ within $T$. By Lemma~\ref{lem:generalized-rotation-digraph}, $D_Y(M_X,T)$ has a cycle, hence a) does not hold.

Suppose (ii) holds. Then we follow an argument symmetric to the previous case and conclude that a) does not hold. 

Suppose that (iii) holds.  
Note that there exists $xy \in M \setminus T'$. Else, since $M$ is not blocked by any matching from $\Ss(T')$ and $T'$ is a stable table, we have $M \in {\cal S}(T')$, a contradiction. Since $M_x\succ_X M \succ_X M_Y$, we must have $M_X(x) >_x y >_x M_Y(x)$. Pick the rotation $\rho\in R_X(T')$ in the form~\eqref{eq:rotation-marriage} such that $(x=x_k, y_k), (x_{k+1}, y_{k+1}) \in \rho$ satisfy $y_k >_x y >_x y_{k+1}$. 
Let $M'=M_X/R(\rho)$. Then $\rho$ is exposed in $M'$ within $T'$ and in addition $M_X\succeq_X M' \succ_X M_Y$. Let $M''=M'/\rho$ and $M^1=(M'\wedge M)\vee M''$. 

We claim that $M'\succ_X M^1 \succ_X M''$. First observe that $M' \succeq_X M^1\succeq_X M''$ because of the definitions of join and meet operators and $M'\succ_X M''$. Moreover, $M^1$ must be distinct from $M'$ or $M''$, since $M'(x) >_x M^1(x)=M(x) >_x M''(x)$. The claim follows. By hypothesis, $\{M\}\cup {\cal S}(T')$ is internally stable and $T'$ is a stable table. Thus, $\{M \} \cup \Ss(T')\subseteq \Ss(T'\cup M)$ and $T' \cup M$ is a stable table. Hence, by Lemma~\ref{lem:internally-stable-are-stable}, $M, M', M''\in\Ss(T'\cup M)$. Using Theorem~\ref{thm:X-optimal-matching}, we deduce that $M^1\in \Ss(T'\cup M)$. Since $T' \cup M$ is a stable table, we deduce that ${\cal S}(T') \cup \{M^1\}$ is internally stable. We can therefore apply Lemma \ref{mid_dissect_rotation}, part 2 (with $M^0=M'$, $M^*=M''$, $M^1=M^1$) and conclude that there exists a generalized rotation $\rho^g\neq \rho$ exposed in $M'$ within $\widetilde{T}$. By Lemma~\ref{lem:generalized-rotation-digraph}, $D_X(M^0,\widetilde T)$ has at least two cycles. The correctness of Algorithm~\ref{alg:DR} (proved in Theorem~\ref{thm:internally_closed_characterization}) implies that $\DR(T,T',\rho)$ does not return $\{\rho\}$. Hence, b) does not hold.

\subsection{Computing an internal closure of an internally stable set}

By building on Theorem \ref{thm:internally_closed_characterization}, Algorithm~\ref{alg_internal_closure} takes as input a marriage instance $T$, a stable table $\widetilde{T}\subseteq T$, and iteratively expand $\widetilde{T}$ to a table $T'$ such that ${\cal S}(T')$ is a closure of ${\cal S}(\widetilde{T})$. In detail, at every step, the algorithm either dissects a rotation from $R_X(T')$ using Algorithm~\ref{alg_dissect_all} as a subroutine, or it finds, for $Z \in \{X,Y\}$, a generalized rotation exposed in $D_Z(M^0,T')$, where $M^0$ is the $Z$-optimal stable matching in $T'$. Note that $T$ and $T'$ are defined as global variables and, for a set $Q$, \emph{$Q$.pop} outputs and simultaneously removes one element of $Q$, and \emph{$Q.enqueue(R')$} adds set $R'$ to $Q$ (the order in which elements are added to or removed from $Q$ does not matter).

\begin{algorithm}[h!]
\caption{Dissect\_All$(T,T',R)$}
\DontPrintSemicolon
\KwIn{Marriage instance $T$, stable subtable $T'\subseteq T$, set $R\subseteq R_X(T')$}
\KwOut{$T''$ stable such that $T'\subseteq T''\subseteq T$}
\begin{algorithmic}[1]
\STATE Let $Q=R$
\WHILE{$Q\neq \emptyset$}
\STATE Let $\rho=Q.$pop
    \STATE Let $R'=Dissect\_Rotation(T,T',\rho)$
    \IF{$|R'| \geq 1$}
    \STATE Set $T'=T'\cup \{\rho| : \rho \in R'\}$ and $Q=Q.enqueue(R')$
    \ENDIF
\ENDWHILE

\end{algorithmic}\label{alg_dissect_all}
\end{algorithm}

\begin{algorithm}[h!]
\caption{Internal\_Closure$(T,\widetilde{T})$}
\DontPrintSemicolon
\KwIn{Marriage instance $T$, stable subtable $\widetilde{T}\subseteq T$}
\KwOut{ $T'\subseteq T$ such that $\Ss(T')$ is an internal closure of $\Ss(\widetilde{T})$}
\begin{algorithmic}[1]
\STATE Set $T'=\widetilde{T}$
\STATE $Dissect$\_$All(T,T',R_X(T'))$

\FOR {$Z \in \{X,Y\}$}
\STATE Let $W =\{X,Y\} \setminus Z$, 
$M^0$ be the $W$-optimal stable matching in $\Ss(T')$, and $D=D_Z(M^0,T)$ 
\WHILE{$D$ has a cycle $C$}
    \STATE Let $\rho$ be the generalized $Z$-rotation corresponding to $C$
    \STATE Set $T' = T' \cup E(\rho)$
    \STATE $Dissect$\_$All(T,T',\{\rho\})$
    \STATE Set $M^0$ as the $W$-optimal matching within $\Ss(T')$, and $D=D_X(M^0,T)$
\ENDWHILE
\ENDFOR 

\STATE Return $T'$
\end{algorithmic}\label{alg_internal_closure}
\end{algorithm}

Note that the output of Algorithm~\ref{alg_internal_closure} is not univocally determined by its input. This is because the internal closure of an internally stable set is, in general, not unique. 

Theorem~\ref{algo:marriage_IStoIC-MC_poly} immediately follows from the following result.

\begin{theorem}\label{algorithmic_correctness}
Let $T$ be a marriage instance with $n$ agents. In time $O(n^4)$, Algorithm~\ref{alg_internal_closure} outputs $T'\subseteq T$ such that $\Ss(T')$ is an internal closure of $\Ss(\widetilde{T})$. 
\end{theorem}
\begin{proof}
Let us first observe that the algorithm is well-defined. Indeed, by Lemma~\ref{lem:dissecting-preseeves-stability} and  Lemma~\ref{lem:gen_rotation_IS}, \emph{Dissect\_Rotation} is always called on a rotation from $R_X(T')$. Moreover, the algorithm halts, since: no entry from $T'$ is ever deleted, each ordered set of edges is added as an element of $Q$ at most once, and when a set of edges is added to $Q$, then the number of entries of $T'$ strictly increases. 
Let therefore $T'$ be an output of the algorithm. We prove that $\Ss(T')$ is internally closed. As our first step, we show that $T'$ is a stable subtable of $T$ containing $\widetilde{T}$. This is clearly true at the beginning of the algorithm, for $T'=\widetilde{T}$ and $\widetilde{T}$ is a stable subtable of $T$. Each time $T'$ is modified, it is enlarged, hence at the end of the algorithm it contains $\widetilde{T}$. Each time we repeat Step 7 in Algorithm~\ref{alg_internal_closure}, $T'$ is internally stable by Lemma~\ref{lem:gen_rotation_IS}, and its counterpart obtained swapping $X$ and $Y$. Because of Lemma~\ref{lem:dissecting-preseeves-stability} and Theorem~\ref{thm:alg:DR:correct}, every call to \emph{Dissect\_All} keeps $T'$ a stable table. 
We conclude therefore that the output $T'$ is a stable subtable of $T$ containing $\widetilde{T}$. 

By repeated applications of Lemma~\ref{lem:dissecting-preseeves-stability} and Lemma~\ref{lem:gen_rotation_IS}, $R_X(T')=R_X(\widetilde{T})\cup R' \cup R''\setminus R_0 $, where $R_0$ are the rotations that are dissected throughout the algorithm, $R'$ is the outcome of their dissections, and $R''$ is the  set of rotations corresponding to cycles in $D_Z(M_Z,T')$ found by the algorithm for $Z \in \{X,Y\}$. All rotations from $R_X(T')$ are checked for dissecting throughout the algorithm. Moreover, by definition, for stable tables $T'\subseteq T''\subseteq T$, and a rotation in $R_X(T')\cap R_X(T'')$, if there is no dissection set for $(T,T',\rho)$, then there is no dissection set for $(T,T'',\rho)$ either. Hence, we can use the generalized rotations characterization from Theorem~\ref{thm:internally_closed_characterization} and conclude that the output $T'$ is internally closed.


As for the running time, it suffices to bound the total running time of all executions of \emph{Dissect\_Rotation} throughout the algorithm. We claim that this is upper bounded by $O(n^4)$. Indeed, for each pair $xy\in 
 T$, when we run \emph{Dissect\_Rotation} on a rotation $\rho$ with $xy \in \rho$, we either not dissect $\rho$, and by construction we never run \emph{Dissect\_Rotation} again on $\rho$ (hence $\rho \in R_X(T')$ for the output $T'$ by Lemma~\ref{lem:dissecting-preseeves-stability} and Lemma~\ref{lem:gen_rotation_IS}), or we dissect $\rho$, and we strictly increase the size of $T'$. Since $T'\subseteq T$ has clearly $O(n^2)$ edges and the number of rotations of a table is linear in the number of edges (see, e.g.,~\cite{Irving}), we deduce that \emph{Dissect\_Rotation} is run $O(n^2)$ times. Since each of its execution takes time $O(n^2)$ by Theorem~\ref{thm:alg:DR:correct}, the bound follows.
\end{proof}

\section{Internally Closed Sets: the Roommate Case}\label{sec:roommate_case}
In this section we return to the more general roommate instances. 
Our major structural contribution is an extension of the theory of rotations to incorporate what we call \emph{stitched rotations}. 
In turn, the complexity of finding a stitched rotation allows us to deduce the computational hardness of deciding whether a set of matchings is internally closed (in this section), as well as deciding whether a set of matchings is vNM stable of deciding if an instance has a vNM stable set of matchings (in the next section), via reductions from $3$-SAT.  


The section is organized as follows. In Section~\ref{sec:irving_algo}, we recall Irving's two-phases algorithms for finding a stable matching (or deducing none exists) in a roommate instance. We then establish several structural properties implied by Irving's algorithm: in Section~\ref{sec:valid_table}, we use its Phase 1 to justify a restriction of the problem domain to \emph{valid} tables; based upon this restriction, in Section~\ref{sec:nonbipartite_rotations} we use Phase 2 of Irving's algorithm to introduce the notion of rotation in the roommate problem which is different from the bipartite case. Rotations in the roommate case can be further categorized into singular and non-singular rotations. In Section~\ref{sec:antipodal_edge_stitched_rotations}, we introduce the new notion of \emph{stitched rotations}, an object stemming from dual rotations which allows us to augment the set of stable matchings of a roommate instance to a larger internally stable set of matchings, or conclude that no such augmentation is possible. Building on all these structural properties, in Subsection \ref{sec:NP_hard_deciding_internally_closed} we employ a reduction from 3-SAT to show the following.

\begin{theorem}\label{thm:coNp-hard-check-closedness}
\texttt{CIC} is co-NP-hard.
\end{theorem}


\subsection{Structural Properties of Solvable Roommate Instances}\label{sec:roommate_structures}

\subsubsection{Irving's Algorithm for the Stable Roommate Problem}\label{sec:irving_algo}

Irving \cite{irving1985efficient} first proposed an
efficient algorithm that, given a roommate instance, either finds a stable matching or concludes that no stable matching exists. Here we discuss Irving's algorithm mostly following the notation from~\cite{Irving}. The algorithm relies on the concept of rotations for the roommate case (see Section~\ref{sec:nonbipartite_rotations} for definitions and properties), and runs in two phases. In Phase 1 described in Algorithm~\ref{alg:Phase-1}, certain edges are eliminated from the original instance. We remark that the ``semiengaged'' relationship is not symmetric -- that is, $x$ being semiengaged to $y$ is different from, and does not imply that, $y$ is semiengaged to $x$.

\begin{algorithm}[H]
\caption{Phase 1 of Stable Roommate Algorithm}
\DontPrintSemicolon
\KwIn{Original preference table $T$}
\KwOut{Reduced (valid) table $T_0$}
\begin{algorithmic}[1]
\STATE Set $T_0=T$ 
\STATE Assign each agent to be free

\WHILE{some free agent $x$ has a nonempty list}
\STATE $y\coloneqq$ first agent on $x$'s list

\IF{some agent $z$ is semiengaged to $y$}
\STATE $y$ rejects $z$, assign $z$ to be free
\ENDIF 
\STATE assign $x$ to be semiengaged to $y$
\STATE delete from $T_0$ edge $x'y$ for all $x'$ such that $x>_y x'$ 
\ENDWHILE
\STATE Remove any empty row and the corresponding agent, and output $T_0$.
\end{algorithmic}\label{alg:Phase-1}
\end{algorithm}

\begin{algorithm}
\caption{Phase 2 of Stable Roommate Algorithm}
\DontPrintSemicolon
\KwIn{Reduced table $T_0$ from Phase 1}
\KwOut{A stable matching or report none exist}
\begin{algorithmic}[1]
\STATE Let $T= T_0$

\WHILE{some list in $T$ has more than one entry and no list in $T$ is empty}
\STATE Find a rotation $\rho$ exposed in $T$
\STATE Set $T=T/\rho$
\ENDWHILE 
\IF{some list in $T$ is empty}
\STATE Report instance unsolvable
\ELSE
\RETURN $T$
\ENDIF
\end{algorithmic}\label{alg:Phase-2}
\end{algorithm}

\subsubsection{Valid Tables}\label{sec:valid_table}
We call a roommate instance $T$ \emph{solvable} if ${\cal S}(T)\neq \emptyset$. Throughout Section \ref{sec:roommate_structures}, we make two assumptions. First, that $T$ is solvable. Second, that each agent in $T$ is matched in some (equivalently, each, see, e.g.,~\cite{Irving}) stable matching.

Although the second hypothesis is not needed, it will greatly simplify the technical details and the notation. Note that both these hypotheses are satisfied by the instances used for our complexity theory results. We now discuss the important concept of \emph{valid} tables\footnote{Those tables are called \emph{stable} in~\cite{Irving}. However, we prefer not to additionally overload the term \emph{stable}.}.




\begin{definition}\label{def:valid_table}
Let $T$ be a roommate instance, and $T'\subseteq T$. $T'$ is a valid table if it satisfies the following conditions: 
\begin{enumerate}[label=\alph*)] 
\item\label{valid_cond_3} No agent's list is empty;
\item\label{valid_cond:first-is-last} Let $z$ be an agent. Then $w=f_{T'}(z)$ if and only if $z=\ell_{T'}(w)$.
\end{enumerate}
$T'$ is moreover called a valid subtable of $T$ if it also satisfies the following property.
\begin{enumerate}
\item[c)]\label{valid_cond_2} Let $zw \in T$. $zw \notin T'$ if and only if $\ell_{T'}(z)>_z w$ or $\ell_{T'}(w)>_w z$.
\end{enumerate}
\end{definition}


\begin{lemma}\label{lem:T-star-is-valid}
Let $T$ be a roommate instance and $T^*=E_S(T)$. Then $T^*$ is a valid table.  
\end{lemma}
\begin{proof}
 a) follows from our assumptions. Given an agent $z$, suppose for a contradiction that $w=f_{T^*}(z)$, and $z>_w \ell_{T^*}(w)$.
Let $M\in \Ss(T)$ such that $w\ell_{T^*}(w)\in M$. Since $M\subseteq T^*$, $w=f_{T^*}(\ell_{T^*}(w)) >_z M(z)$. Hence, $wz$ blocks $M$, a contradiction.

\end{proof}

\subsubsection{Rotations}\label{sec:nonbipartite_rotations}

\paragraph{Known definitions and facts.}

A major difference between the stable marriage problem and the stable roommate problem is the absence of a (known) non-trivial lattice structure in the former. This change calls for different tools for analyzing the problem. Specifically, we adopt a distinct definition of rotations for the roommate instance in Definition \ref{def:rotation}, and then categorize them further into singular and non-singular rotations. Unless otherwise stated, all proofs and concepts from this section can be found in~\cite{Irving}. 

Throughout the rest of the section, we assume that $T'$ is a valid subtable of $T$.

\begin{definition}\label{def:fixed_pair}
We say that $xy\in T$ is a \emph{fixed pair} if $xy\in M$ for all $M\in \Ss(T)$.
\end{definition}

\begin{definition}\label{def:rotation} A sequence
$\rho=(x_{0},y_{0}),(x_{1},y_{1}),...,(x_{r-1},y_{r-1})$ is called
a \emph{rotation exposed in $T'$} if $y_{i}=f_{T'}(x_{i}),y_{i+1}=s_{T}(x_{i})$
for all $i \in [r-1]_0$, where indices are taken modulo $r$\footnote{By condition \ref{valid_cond:first-is-last} from Definition~\ref{def:valid_table}, we can see that $x_{i}=\ell_{T'}(y_{i})$.}.
Call set $X(\rho)=\{x_{i}| i \in [r-1]_0\}$ the \emph{$X$-set of
$\rho$} and $Y(\rho)=\{y_{i}| i \in [r-1]_0\}$ the \emph{$Y$-set of $\rho$}.

If $\rho=(x_{0},y_{0}),(x_{1},y_{1}),...,(x_{r-1},y_{r-1})$ is a
rotation exposed in the table $T'$, then $T'/\rho$ denotes the table
obtained from $T'$ by deleting all pairs $y_{i}z$ such that $y_{i}$
strictly prefers $x_{i-1}$ to $z$. We refer to this process as to the \emph{elimination of $\rho$ in $T$}.
\end{definition}

Similarly to the marriage case, we also interpret $\rho$ as a table with entries $x_iy_i$ for $i \in [r-1]_0$. 
We also write $zw\in\rho$ if $(z,w)$ or $(w,z)$ is among the sequence of pairs in $\rho$.

Let ${\cal Z}(T)$ denote the set of all rotations exposed at some valid subtable of $T$. In the following, we also fix rotations
\begin{equation}\label{eq:rho}\rho=(x_{0},y_{0}),(x_{1},y_{1}),...,(x_{r-1},y_{r-1}).\end{equation}
and $\sigma$ from ${\cal Z}(T)$, with $\rho\neq \sigma$.

We first observe that eliminating a rotation from a table preserves validity and moves the favorite entry of each agent from the $X$-set one position to the right.

\begin{lemma}\label{lem:valid-stay-valid}
Let $\rho$ be a rotation exposed at $T'$ of the form~\eqref{eq:rho}. Then $T' / \rho$ is a valid subtable of $T$ and for each $i\in [r-1]_0$, we have $f_{T'/\rho}(x_i)=y_{i+1}=s_{T'}(x_i)$.
\end{lemma}

On the other hand, unless a valid subtable has exactly one entry per row, it always has an exposed rotation.

\begin{lemma}\label{lem:always-a-rotation}
Suppose that some agent in $T'$ has a preference list of length at least $2$. Then there is a rotation exposed at $T'$.
\end{lemma}

The relationship between rotations and valid tables is even tighter: the former can be used to characterize the latter. Let ${\cal Z}=\{\rho_{1},\rho_{2},...,\rho_{t}\}  \subseteq {\cal Z}(T)$, so that, for $i \in [t]$, $\rho_i$ is exposed in $(((T/\rho_1)/\rho_2)/ \dots ) / \rho_{i-1}$. Let $T/\mathcal{Z}$ denotes the instance obtained from $T$ after iteratively eliminating $\rho_{1},\rho_{2},...,\rho_{t}$ in an appropriate order. 

\begin{lemma}\label{lem:all-valid-table-from-rotations}
$T'=T/{\cal Z}$ for some ${\cal Z}\subseteq {\cal Z}(T)$. 
\end{lemma}

Rotations can be further classified into two
categories: \emph{singular} and \emph{non-singular}.

\begin{definition}\label{rotation_singularity} $\rho$ 
is called a \emph{non-singular} rotation if $\overline \rho \in {\cal Z}(T)$, where \begin{equation}\label{eq:dual-rotation}\overline{\rho}=(y_{1},x_{0}),(y_{2},x_{1}),...,(y_{i},x_{i-1}),...,(y_{0},x_{r-1}),\end{equation} and \emph{singular} otherwise.
If $\rho$ is non-singular, we call $\overline{\rho}$ the dual
of $\rho$. Observe that $\overline {\rho}$ is also non-singular, and $\overline{\overline{\rho}}=\rho$. The subset of ${\cal Z}(T)$ containing all singular (resp.~non-singular) rotations is denoted by ${\cal Z}_s(T)$ (resp.~${\cal Z}_{ns}(T)$).
\end{definition}

It turns out that one can eliminate singular rotations before eliminating any non-singular one. 

\begin{lemma}
\label{lem:singular-first} 
$T/{\cal Z}_s$ is well-defined, that is, we can order singular rotations of $T$ so that they can be iteratively eliminated, starting from $T$. Moreover, for every stable matching $M$, $M\subseteq T/\mathcal{Z}_{s}$.
\end{lemma}

The next results highlights a tight connection between rotations, valid tables, and stable matchings. 

\begin{theorem}\label{thm:T_lem:eliminating-rotations-preserves-exposition} Any $M \in {\cal S}(T)$ is a valid subtable  of $T$ resulting from a sequence of rotation eliminations from $T$, i.e., $M=T/\mathcal{Z}_M$ for some $\mathcal{Z}_M\subseteq \mathcal{Z}(T)$. Moreover, $\mathcal{Z}_M$ contains every singular rotation and exactly one of each dual pair of rotations and is univocally determined by $M$. Conversely, every sequence of rotation elimination will eventually lead to a valid subtable with one entry per row corresponding to a stable matching.
\end{theorem}

Following the previous theorem, for $M\in\mathcal{S}(T)$, we let $\mathcal{Z}_{M}(T)$ denote the set of rotations in table $T$ such that $M=T/\mathcal{Z}_{M}(T)$.

\begin{lemma}\label{lem:stable-pairs} Let $xy \in T$. $xy $ is a stable pair if and only if one of the following holds: (i) there exists $\rho \in {\cal Z}_{ns}(T)$ such that $xy \in \rho$; (ii) $xy$ is a fixed pair. 
\end{lemma}

The next lemmas present further known properties of rotations.

\begin{lemma}\label{lem:singular_expose} $\rho$ is singular if and only if there is a valid subtable of $T$ in which $\rho$ is the only exposed rotation. 
\end{lemma}


\begin{lemma}\label{lem:eliminating-rotations-preserves-exposition} Suppose that $\rho$  and $\sigma$ are exposed in $T'$. Then either $\rho$ is exposed in $T'/\sigma$, or $\sigma=\overline{\rho}$.
\end{lemma}


\begin{lemma}\label{lem:only-rho-and-dual}
Suppose $\rho$ is non-singular. If both $\rho$ and its dual $\overline{\rho}$ are exposed in $T'$, then, for each $i \in [r-1]_0$, the list of $x_i$ in $T$ contains only $y_i$ and $y_{i+1}$, and the list of $y_i$ in $T$ contains only $x_{i-1}$ and $x_i$. The elimination of $\rho$ or $\overline{\rho}$ from $T'$ reduces the list of each $x_i$ and each $y_i$ to a single entry.
\end{lemma}

\paragraph{A new fact.} Next, we prove an additional property of rotations that will be useful for our analysis. Recall that we assume that $T$ is a solvable roommate instance and that $\rho \in {\cal Z}(T)$ has the form~\eqref{eq:rho}.

\begin{lemma}\label{sm_dual}
Let $\rho$ be non-singular. There exists $M_1,M_2 \in {\cal S}(T)$ such that: $x_iy_i\in M_1$ and $x_{i-1}y_i\in M_2$ for all $i \in [r-1]_0$; the symmetric difference $G[M_1\triangle M_2]$ contains only singletons and one even cycle, and the latter contains all and only the nodes in the X-set and Y-set of $\rho$.
\end{lemma}
\begin{proof}
By definition of rotation, there exists a full execution of Phase 1, followed by a partial execution of Phase 2 leading from $T$ to a table $T''$ such that rotation $\rho$ is exposed in $T''$. By~Lemma~\ref{lem:singular_expose}, there must exist some rotation other than $\rho$ that is exposed in $T''$. 

\smallskip

\emph{Case a): $\overline{\rho}$ is exposed in $T''$}. By Lemma~\ref{lem:only-rho-and-dual}, for $i \in [r-1]_0$, the preference lists of $x_i$ contains $y_i$ and $y_{i+1}$ only, and the preference list of $y_i$ contains $x_{i}$ and $x_{i-1}$ only. 

Suppose first $\rho, \overline{\rho}$ are the only rotations exposed in $T''$. We claim that every agent $z$ that does not belong to the $X$-sets of $\rho, \overline{\rho}$ has a single element in their preference list. If this is not the case, then $T''/\rho$ also contains an agent $z$ with two agents in their preference list. By Lemma~\ref{lem:always-a-rotation}, there is a rotation $\pi$ exposed at $T''/\rho$. Since each agent $x_i, y_i$ with $i \in [r-1]_0$ has exactly one agent in their preference list in $T''$ (again by Lemma \ref{lem:only-rho-and-dual}), the $X$-set and the $Y$-set of $\pi$ do not intersect $\{x_i,y_i\}_{i \in [r-1]_0}$. Hence, $\pi$ is also exposed in $T''$, a contradiction. Thus, the preference lists of $T''$ are as follows: agents in the $X$-sets of $\rho$ and $\overline{\rho}$ have preference lists of length $2$; every other agent has a preference lists of length $1$. Hence, $T''/\rho$ (resp.~$T''/\overline{\rho}$) corresponds to a stable matching $M_1$ (resp.~$M_2$) by Theorem~\ref{thm:T_lem:eliminating-rotations-preserves-exposition}. It is easy to check that $M_1$ and $M_2$ satisfy the property required by the thesis of the lemma.

Hence, suppose that there is a rotation $\pi \neq \rho,\overline {\rho}$ exposed in $T''$. By Lemma~\ref{lem:eliminating-rotations-preserves-exposition}, both $\rho$ and $\overline{\rho}$ are exposed in $T''/\pi$. We can therefore iterate the argument on $T''/\pi$, eventually obtaining a valid subtable $T''$ where $\rho$ and $\overline{\rho}$ are the only exposed rotations. This case has been already investigated above.

\smallskip

\emph{Case b): $\overline{\rho}$ is not exposed in $T''$.} Let $\sigma\neq \rho$ be any rotation exposed in $T''$. By Lemma~\ref{lem:eliminating-rotations-preserves-exposition}, $\rho$ is also exposed in $T''/\sigma$. Hence, we can iterate the argument on $T''/\sigma$. Since we know by Lemma~\ref{lem:singular_expose} that $\rho$ cannot be the only rotation exposed in $T''$, we eventually, obtain that $\overline{\rho}$ is exposed in $T''$. We are thus in Case a) above. \end{proof}

\begin{definition}\label{def:disjoint_rotations}
We say that a set of rotations $\rho^1,\rho^2,\dots,\rho^k$ where $\rho^q=(x_1^q,y_1^q), \dots (x_{n_q}^q,y_{n_q}^q)$ for $q \in[k]$ are \emph{disjoint} if the sets of pairs contained in each rotation are disjoint. In other words, $\{x_1^{q_1}y_1^{q_1}, \dots ,x_{n_{q_1}}^{q_1}y_{n_{q_1}}^{q_1}\}\cap \{x_1^{q_2}y_1^{q_2}, \dots ,x_{n_{q_2}}^{q_2}y_{n_{q_2}}^{q_2}\}=\emptyset$ for any $1\leq q_1 < q_2 \leq k$.
\end{definition}

Similarly to the marriage case, for $\rho$ as in~\eqref{eq:rho}, we let $E(\rho) = \cup_{i\in[r-1]_0}\{ x_iy_{i}, x_{i}y_{i+1} \} $, with indices taken modulo $r$ as usual.

\subsubsection{Antipodal edges and stitched rotations}\label{sec:antipodal_edge_stitched_rotations}

Throughout this subsection, we keep assuming that $T$ is a solvable roommate instance and let $T^*=E_S(T)$ be the subtable of $T$ containing all and only the stable edges. Note that $E_S(T^*)=T^*$ and that each $T'$ with $T^*\subseteq T' \subseteq T$ is also solvable. In particular, all properties of valid tables developed in Section \ref{sec:nonbipartite_rotations} for subtables of $T$ apply to subtables of $T^*$.

\paragraph{Antipodal Edges.}
Recall that $\Ss(T^*)$ is internally stable by Lemma \ref{lem:internally-stable-are-stable}.
We first argue that to ``expand" ${\cal S}(T^*)$ to a strictly larger  internally stable set of matchings, i.e., to find a  stable table $T'\supseteq T^*$ such that $\Ss(T')\supsetneq \Ss(T^*)$, any additional edge $e\in T'\setminus T^*$ must satisfy what we call the \emph{antipodal} condition, as defined below. 
\begin{definition}
We say that an edge $e=xy\in T\setminus T^{*}$ satisfies the antipodal
condition (wrt $T^{*}$) if exactly one of the following
is true:
\begin{itemize}
\item $ y >_x f_{T^{*}}(x)$ and $x <_y  \ell_{T^{*}}(y)$, or 
\item $x >_yf_{T^{*}}(y)$ and $ y >_x \ell_{T^{*}}(y)$. 
\end{itemize}
\end{definition}

\begin{lemma}\label{lem:wlog_antipodal}
Let $e \in T\setminus T^*$. Assume that $e$  is not an antipodal edge. Then every $M \in {\cal M}(T)$ with $e \in M$ is blocked by some edge from $T^*$.  
\end{lemma} 
\begin{proof}
Let $xy\in T\setminus T^{*}$, and $M\in{\cal M}(T)$ such that $xy \in M$.  Suppose that $xy$ is not antipodal. Consider the different cases below.

1. $y>_x f_{T^{*}}(x)$ and $x >_y \ell_{T^{*}}(y)$. Let $M_s \in {\cal S}(T)$ such that $y \ell_{T^*}(y) \in M_s$, which exists by definition of $T^*$. Then $M_s(x) \leq_x f_{T^*}(x) <_x y$. Thus, $xy$ blocks $M_s$, a contradiction.

2. $y<_xf_{T^{*}}(x)$ and $x<_y \ell_{T^{*}}(y)$. Let $M_s \in {\cal S}(T)$ such that $x f_{T^*}(x) \in M_s$, which exists by definition of $T^*$. Then $M_s(y)>_y \ell_{T^*}(y)>_y x$. Thus, $xy$ is an irregular edge of $M\cup M_s$. By Lemma~\ref{symmetric_diff}, we conclude that either $M$ blocks $M_s$, or $M_s$ blocks $M$. The former cannot happen by stability of $M_s$, hence the latter happens. Since $M_s\subseteq T^*$, the thesis follows. 

3. $ x >_y f_{T^{*}}(y)$ and $ y >_x \ell_{T^{*}}(x)$; or $ x <_y f_{T^{*}}(y)$ and $ y <_x \ell_{T^{*}}(x)$. Symmetric to cases 1 and 2 above, respectively, swapping the roles of $x$ and $y$.

4. $ f_{T^*}(x)>_x y >_x \ell_{T^*}(x)$ and $ f_{T^*}(y) >_y x >_y \ell_{T^*}(y)$. We first claim that $xy$ is not removed from $T$ during Phase $1$ of Irving's algorithm. Indeed, we have $ x >_y  \ell_{T^*}(y) \geq_y \ell_{T_0}(y)$ where $T_0$ is the table output by Phase 1, and similarly $y >_x \ell_{T^*}(x) \geq_x \ell_{T_0}(x)$. The edges $x'y'$ removed during Phase 1 conversely verify at least one of $y'<_{x'} \ell_{T_0}(x')$ and $ x' <_{y'} \ell_{T_0}(y')$. 
    
Hence, $xy \in T_0$. However, since $xy$ is not a stable pair, it is removed from $T_0$ during all executions of Phase 2. By Lemma~\ref{lem:singular-first}, there is a valid execution of Phase 2 where all singular rotations are eliminated first, as to obtain table $T'$. Using again Lemma~\ref{lem:singular-first}, we know that $T^*\subseteq T'$. Hence $ \ell_{T'}(x) \leq_x \ell_{T^*}(x) <_x y$ and $ \ell_{T'}(y) \leq_y \ell_{T^*}(y) <_y x$. However, pairs $x'y'$ that have been removed when going from $T_0$ to $T'$ satisfy, similarly to the above, at least one of $y'<_{x'}\ell_{T'}(x')$ and $ x'<_{y'} \ell_{T'}(y')$. Thus, $xy \in T'$. 

Now fix an execution of Phase 2 that eliminates all singular rotations first, and consider the last valid table $T''$ of this execution of Phase 2 that contains $xy$. That is, $xy \in T''$ but $xy \notin T''/ \rho$ for some rotation $\rho$ exposed in $T''$. Let $\rho$ have the form~\eqref{eq:rho} and note that $\rho$ is non-singular, because all singular rotations have been eliminated before producing table $T'$. Then there exists $i \in [r-1]_0$ such that $x=y_i$ or $y=y_i$. Suppose w.l.o.g.~that the latter happens. Hence, $x_{i-1} >_y x >_y x_{i}$, where the first relation comes from the description of the edges that are removed during a rotation elimination (see Algorithm~\ref{alg:Phase-2} and Definition~\ref{def:rotation}), and the second from the fact that $\ell_{T''}(y)=x_i$ and $x_iy$ is a stable pair by Lemma~\ref{lem:stable-pairs} while $xy$ is not by hypothesis. Let $y'=f_{T''/\rho}(x)$. 

\begin{claim}\label{cl:1} $xy'$ is a stable pair. 
\end{claim} 

\noindent \underline{Proof of Claim}. By Definition~\ref{def:valid_table}, we know that $x=\ell_{T''/\rho}(y')$. If there is a sequence of rotation eliminations leading from $T''/\rho$ to a table $T'''$ with one entry per agent such that $xy' \in T'''$, we deduce that $T'''$ is a stable matching by Theorem~\ref{thm:T_lem:eliminating-rotations-preserves-exposition}. 
Hence, suppose this is not the case. Then, there must be some iteration where $xy'$ is deleted. Since $y'=f_{T''/\rho}(x)$, this deletion happens when a rotation $\sigma=(x_0',y_0'),(x_1',y_1'),\dots,(x'_{r'-1},y'_{r'-1})$ is eliminated, where, modulo a shifting of the indices, $y'=y'_1$. Let $T'''$ the table before the elimination of $\sigma$. By definition of rotation, it must be that $x'_1=\ell_{T'''}(y')$, hence $x'_1=x$ since $x=\ell_{T''}(y')$, $T''/\rho\supseteq T'''$ and $xy' \in T'''$. Moreover, $\sigma$ is a non-singular rotation, since all singular rotations have been eliminated before obtaining $T''/\rho$, and the statement follows by Lemma~\ref{lem:stable-pairs}. $\hfill \diamond$

\begin{claim}\label{cl:2}
$y'>_x y$.
\end{claim}
\underline{Proof of Claim}. By Claim~\ref{cl:1}, $xy'$ is a stable pair, while $xy$ is not. Hence, $y'\neq y$. Recall that $f_{T''}(x)>_x y$. Hence, if some entry preceding $y$ in the preference list of $x$ in $T''$ is not eliminated when eliminating $\rho$, we have $y'=f_{T''/\rho}(x)>_x y$. Hence, assume that by eliminating $\rho$ we delete all entries in the preference list of $x$ before $y$. To delete $f_{T''}(x)$, we must have $x=x_{j}$ for some $j \in [r-1]_0$ (recalling that  $\rho$ has the form~\eqref{eq:rho}). Then $f_{T''/\rho}(x)=s_{T''}(x)$ by Lemma~\ref{lem:valid-stay-valid}. Since $xy$ is deleted when eliminating $\rho$, we have $y'=s_{T''}(x)>_x y$, as required.  $\hfill \diamond$

\begin{claim}\label{cl:3} There exists $M_s \in {\cal S}(T)$ such that $M_s(x)=y'$ and $M_s(y)>_y x$.\end{claim}

\noindent \underline{Proof of Claim}. 
By construction, $y'=f_{T''/\rho}(x)$. By construction and Lemma~\ref{lem:singular_expose}, starting from $T''/\rho$, there is always an exposed rotation whose elimination does not delete $xy'$. 
Similarly to the proof of Lemma~\ref{sm_dual}, we can iteratively rotate such exposed rotation ending up in a table $T'''$ with no rotation exposed and $xy' \in T'''$. Theorem~\ref{thm:T_lem:eliminating-rotations-preserves-exposition}, $M_s$ is stable. By construction $M_s(x)=y'$ and $M_s(y)\geq_y\ell_{T''/\rho}(y)>_y x$, as required. $\hfill \diamond$

\smallskip

Let $M_s$ be the stable matching of $T$ whose existence is guaranteed by Claim~\ref{cl:3}. Thus, $M_s(y)>_y x$.  Using also Claim~\ref{cl:2}, we know that $M_s(x)=y'>_xy$. Hence, $xy$ is an irregular edge of $M \cup M_s$, and the thesis follows similarly to case 2 above.
\end{proof}



\smallskip

\paragraph{Stitched Rotations.}
We introduce the notion of stitched rotations, an object that allows us to assemble antipodal edges and expand the set of stable matchings to a larger internally stable set.
\begin{definition}
\label{def:stitched_rotation} Let $\rho$ be as in~\eqref{eq:rho}. Suppose $\rho \in {\cal Z} (T^* \cup \rho)$ is exposed in $T^* \cup \rho$ and such that $x_iy_{i}$ satisfies the
antipodal condition with respect to $T^{*}$ for all $i \in [r-1]_0$. We call $\rho$ \emph{a }\emph{stitched 
rotation} with respect to $T^{*}$ if $\rho \in {\cal Z}_{ns}(T^{*}\cup \rho)$.
\end{definition}

\begin{figure}
{\small{\begin{center}
 $T=$ \begin{tabular}{||c c||} 
 \hline
 Person & Preference list \\ [0.5ex]
 \hline
 $x_1$ & $x_2$ $x_3$   \\ 
 \hline
 $x_2$ & $x_4$ $x_1$   \\ 
 \hline
 $x_3$ & $x_5$ $x_1$ $x_4$ \\ 
 \hline
 $x_4$ & $x_3$ $x_6$ $x_2$ \\ 
 \hline
 $x_5$ & $x_6$ $x_7$ $x_3$ \\ 
 \hline
 $x_6$ & $x_4$ $x_8$ $x_5$ \\ 
 \hline
 $x_7$ & $x_5$ $x_8$ \\ 
 \hline
 $x_8$ & $x_7$ $x_6$ \\ 
 \hline
\end{tabular}
\, $T^*=$ \begin{tabular}{||c c||} 
 \hline
 Person & Preference list \\ [0.5ex] 
 \hline
 $x_1$ & $x_2$   \\ 
 \hline
 $x_2$ & $x_1$   \\ 
 \hline
 $x_3$ & $x_4$ \\ 
 \hline
 $x_4$ & $x_3$ \\ 
 \hline
 $x_5$ & $x_6$ $x_7$ \\ 
 \hline
 $x_6$ & $x_8$ $x_5$ \\ 
 \hline
 $x_7$ & $x_5$ $x_8$ \\ 
 \hline
 $x_8$ & $x_7$ $x_6$ \\ 
 \hline
\end{tabular}
\, $T^1=$  \begin{tabular}{||c c||} 
 \hline
 Person & Preference list \\ [0.5ex] 
 \hline
 $x_1$ & $x_2$ $\boxed{x_3}$ \\ 
 \hline
 $x_2$ & $\boxed{x_4}$ $x_1$ \\ 
 \hline
 $x_3$ & $\boxed{x_1}$ $x_4$ \\ 
 \hline
 $x_4$ & $x_3$ $\boxed{x_2}$ \\ 
 \hline
 $x_5$ & $x_6$ $x_7$ \\ 
 \hline
 $x_6$ & $x_8$ $x_5$ \\ 
 \hline
 $x_7$ & $x_5$ $x_8$ \\ 
 \hline
 $x_8$ & $x_7$ $x_6$ \\ 
 \hline
\end{tabular}
\end{center}}}
\caption{The instance from Example~\ref{ex:roommate-stitched}.  Entries $(x_i,y_i) \in \rho$ are boxed in $T^1$.}\label{fig:ex:roommate-stitched}
\end{figure}

\begin{example}\label{ex:roommate-stitched}

Consider the instance of the stable roommate problem described by the table $T$ in Figure~\ref{fig:ex:roommate-stitched}, left.  Then the set of all stable matchings of $T$ consists of $\{M_1,M_2\}$ where $M_1=\{x_1x_2,x_3x_4,x_5x_6,x_7x_8\}$ and $M_2=\{x_1x_2,x_3x_4,x_5x_7,x_6x_8\}$, giving table $T^*$ from Figure~\ref{fig:ex:roommate-stitched}, center. We can find the stitched rotation $\rho=(x_1,x_3), (x_2,x_4)$ with respect to table $T^*$. Augmenting $T^*$ with $\rho$ yields table $T^1=T^* \cup \rho$ from Figure~\ref{fig:ex:roommate-stitched}, right. One can check that no other stitched rotation exists in $T^1$. Hence, $\Ss(T^1)$ is an internally closed set of matchings (see Theorem~\ref{augment_dual}).
\end{example}

Next lemma gives a necessary and sufficient condition for $\rho$ to be a stitched rotation.

\begin{lemma}
\label{lem:stitched_rotation_paird_matching} Let $\rho=(x_{0},y_{0}),(x_{1},y_{1}),\dots,(x_{r-1},y_{r-1})\subseteq T$
be a rotation of ${\cal Z}(T^* \cup \rho)$ exposed in $T^* \cup \rho$. Then $\rho$ is a stitched rotation with respect to $T^{*}$ if and only if $T^* \cap \rho=\emptyset$ and there exists $M\in\Ss(T^{*})$ such that $x_{i}y_{i+1}\in M$ for $i \in [r-1]_0$.
\end{lemma}

\begin{proof}
First observe that $T^*,T^*\cup\rho$ are solvable instances, since $E_S(T)=T^*\subseteq T^*\cup\rho \subseteq T$. Suppose first that $\rho$ as in~\eqref{eq:rho} is a stitched rotation. Lemma~\ref{sm_dual} guarantees the existence of $M \in {\cal S}(T^* \cup \{\rho\})$ such that $x_iy_{i+1} \in M$ for $i \in [r-1]_0$. We need to show that $M \in {\cal S}(T^*)$. By construction, $M \in {\cal M}(T^*)$. Moreover, since $M$ is not blocked by any edge of $T^* \cup \rho$, it is not blocked by any edge of $T^*$. Hence $M \in {\cal S}(T^*)$.

Conversely, assume that $ T^* \cap \rho=\emptyset$ and there exists $M\in\Ss(T^{*})$ such that $x_{i}y_{i+1}\in M$ for $i \in [r-1]_0$.
Since $\rho$ is exposed in $T^* \cup \rho$, we have that for $i \in [r-1]_0$, $y_{i}=f_{T^*\cup \rho}(x_i)$ and $x_{i}=\ell_{T^*\cup \rho}(y_i)$. Since $x_iy_i \notin T^*$, we deduce that $x_iy_i$ is antipodal.

Let $T'$ be table where each agent's list has only its partner in $M$. By Theorem~\ref{thm:T_lem:eliminating-rotations-preserves-exposition}, we know that $T'$ can be obtained from $T^*$ by iteratively eliminating some exposed rotations $\rho_1,\rho_2,\dots,\rho_k\in \mathcal{Z}(T^*)$ in this order. We claim that $\rho_1,\rho_2,\dots,\rho_k\in \mathcal{Z}(T^*\cup \rho)$, and can be iteratively eliminated, in the same order, as exposed rotations starting from $T^*\cup \rho$, without removing any of the edges $x_iy_i$. Once this is is proved, observe that the resulting table of the latter sequence of rotation elimination is $T'\cup\rho$. One easily observes that $\rho$ and its dual $\overline \rho$ are then rotations exposed in $T'\cup \rho$ since by definition of rotation, for $i \in [r-1]_0$ we have $y_i>_{x_i} y_{i+1}$ and $x_i<_{y_i}x_{i-1}$. Hence, $\rho \in {\cal Z}_{ns}(T^* \cup \rho)$ is a stitched rotation, concluding the proof. 

Let $i\in [r-1]_0$. Since $y_{i}=f_{T^*\cup \rho}(x_i)$ and $y_{i+1}=s_{T^*\cup \rho}(x_i)=f_{T^*}(x_i)$, the first rotation $\rho_j \in \{\rho_1,\dots,\rho_k\}$ that counts $x_i$ in its $X$-set must have $(x_i,y_{i+1}) \in \rho_j$. When eliminated from $T^*$, 
by definition it eliminates edge $x_{i}y_{i+1}$. But $x_i y_{i+1}\in T'$, a contradiction. Hence, $x_i$ is not in the $X$-set of $\rho_j$.

Since $T^*$ is a valid table by Lemma~\ref{lem:T-star-is-valid}, we know that $\ell_{T^*}(y_{i+1})=x_i$. Hence, if $\rho_j \in \{\rho_1,\dots,\rho_k\}$ is the first rotation that counts $y_{i+1}$ in its $Y$-set, the elimination of $\rho_j$ eliminates $x_i y_{i+1} \in M'$, again a contradiction. 
Hence for $j \in [k]$ the $X$-set (resp.~$Y$-set) of $\rho_j$ is disjoint from the $X$-set (resp.~$Y$-set) of $\rho$. Hence, we can eliminate $\rho_1,\dots,\rho_k$ in this order from $T^* \cup \rho$. To observe that, for $i \in [r-1]_0$, $x_iy_i$ is not removed during this sequence of elimination, observe that $y_i=f_{T^*\cup\rho}(x_i)$ and $y_i$ is not in any $Y$-set of $\rho_1,\dots,\rho_k$. 
\end{proof}





The next theorem shows that the stable subtable $T^*$ of a solvable instance is internally closed if and only if there is no stitched rotation w.r.t.~$T^*$.

\begin{theorem}\label{augment_dual} 1. Let $M\in\M(T)\backslash \Ss(T^*)$, and assume that $\{M\}\cup \Ss(T^*)$ is internally stable. Then there exists a stitched rotation $\rho$ w.r.t.\ $T^*$. 

2. Conversely, if $\rho$ is a stitched rotation w.r.t.\ $T^*=E_S(T)$, then there exists a matching $M \in {\cal M}(T)\setminus {\cal S}(T^*)$ with $\rho \subseteq M$ so that $\{M\} \cup {\cal S}(T^*)$ is internally stable. 
\end{theorem}
\begin{proof}
1. 

\begin{claim}\label{cl:again-stable-table}
$M\cup T^*$ is a stable table. \end{claim} 

\noindent \underline{Proof of Claim}. By definition, $M$ is not blocked by any edge in $M\cup T^*$, so $M\in \Ss(M\cup T^*)$. Since $T^*=E_S(T^*)$, we have $M\cup\{M':M'\in\Ss(T^*)\}=M\cup T^*$. By Lemma \ref{lem:internally-stable-are-stable}, the internal stability of $\{M\}\cup \Ss(T^*)$ implies $\{M\}\cup \Ss(T^*)\subseteq \Ss(M\cup T^*)$, which implies that $M\cup T^*$ is a stable table. 
$\hfill \diamond$

\smallskip 

We next show how to construct the stitched rotation $\rho$ whose existence is claimed in the thesis. Since $M \notin {\cal S}(T^*)$ and $E_S(T^*)=T^*$, we must have $M\setminus T^* \neq \emptyset$. Take therefore any $x_0y_0\in M\backslash T^*$. By Lemma~\ref{lem:wlog_antipodal}, $x_0y_0$ satisfies the antipodal condition. We assume without loss of generality that $y_0>_{x_0}f_{T^*}(x_0)$. We claim that there exists $M^* \in \Ss(T^*)$ such that $M^*(x_0)=f_{T^*}(x_0)$ and $M^*$ and $M$ do not block each other. Any matching contained in $T^*$ does not block $M$ by hypothesis. Moreover, there is a matching $M^* \in {\cal S}(T)\subseteq \Ss(T^*)$ such that $M^*(x_0)=f_{T^*}(x_0)$ by definition of $T^*$. Since $M \subseteq T$, we deduce that $M$ does not block $M^*$.

Define $y_1=M^*(x_0)$, $x_1=M(y_1)$. By Lemma~\ref{symmetric_diff}, $x_0>_{y_1} x_1$. As $y_1=f_{T^*}(x_0)$, we have that $x_0=\ell_{T^*}(y_1)$ since $T^*$ is a valid table by Lemma~\ref{lem:T-star-is-valid}. Hence, $\ell_{T^*}(y_1)>_{y_1} x_1$, i.e., $x_1y_1\in M\setminus T^*$. Since $M \cup T^*$ is a stable table by Claim~\ref{cl:again-stable-table}, there exists matching $M^{**}$ such that $M^{**}(x_1)=f_{T^*}(x_1)$ and $M^{**}$ and $M$ do not block each other. We deduce $y_1>_{x_1}f_{T^*}(x_1)$. We can therefore iterate this argument, obtaining edges $x_3y_3, x_4y_4, ... \in M\setminus T^*$, with the properties that, for $i \in \N$, $x_i=\ell_{T^*}(y_{i+1})$ (or, equivalently, $y_{i+1}=f_{T^*}(x_{i})$), $y_i >_{x_i}f_{T^*}(x_{i})$ and $x_i=M(y_i) <_{y_i}\ell_{T^*}(y_i)$. Hence, in the sequence $y_0,x_0,y_1,x_1,y_3,...$ each node prefers the one that precedes it. 

Let $z$ be the first node that appears twice in this list. We claim that $z=y_0$. Assume first $z=x_i=y_j$ for some $i,j$. If $i \geq j$, then there are two edges of $M$ incident to $z$, a contradiction. If instead $i<j$ then the two edges $x_i y_i, x_j y_j$ are both in $M$ and incident to $z$. However, by construction $y_i >_z f_{T^*}(z) \geq_z \ell_{T^*}(z) >_z x_j$, showing that those edges are distinct hence contradicting that $M$ is a matching. If $z=x_i=x_j$, then again we contradict that $M$ is a matching. Hence, $z=y_i=y_j$ with $i<j$. If $i \neq 0$, then $f_{T^*}(x_{i-1})=z=f_{T^*}(x_{j-1})$, contradicting that $T^*$ is a valid table. Hence, $z= y_0=y_r$. We deduce therefore the following:

\begin{claim}\label{cl:rho-X-Y-disjoint}
$\rho=(x_0,y_0),(x_1,y_1),\dots,(x_{r-1},y_{r-1})$ belongs to $\mathcal{Z}(T^*\cup \rho)$ and is exposed in $T^*\cup\rho$. Moreover, its $X$- and $Y$-sets are disjoint, and $T^*\cap \rho = \emptyset$.
\end{claim}


To show that $\rho$ is a stitched rotation w.r.t.~$T^*$, we derive more properties of the set $M\setminus T$. Suppose there exists $xy \in M\setminus (T^*\cup \rho)$. We can repeat the argument used above to construct $\rho$ and obtain a rotation $\rho'$ whose first pair is, without loss of generality, $(x,y)$. 

\begin{claim}\label{cl:rho-rho'-disjoint}
$\rho$ and $\rho'$ are disjoint.
\end{claim}

\noindent \underline{Proof of Claim}. Suppose by contradiction that they are not, and let $z$ be the first agent (according to the order of $\rho'$ starting from $(x,y)$, where in each pair the $X$-agent precedes the $Y$-agent) that is in both $\rho$ and $\rho'$. Note that, by construction, $zM(z) \in \rho \cap \rho'$.  Hence $z$ is an $X$-agent of $\rho'$. Suppose first it is an $X$-agent of $\rho$ as well. Note that $z\neq x$, by definition of $x$. Let therefore $(\bar x, \bar y)$ be the pair that precedes $(z,M(z))$ in $\rho'$. By construction, $\bar x=\ell_{T^*}(M(z)) \in \rho$, contradicting the choice of $z$. Last, assume that $z$ is a $Y$-agent of $\rho$. Then $(z,M(z))\in \rho'$ and $(M(z),z) \in \rho$. However, from the properties deduced while constructing $\rho$, we have $M(z)<_{z} \ell_{T^*}(z)\leq_z f_{T^*}(z)<_z M(z)$, a contradiction. $\hfill \diamond$

\smallskip 

The construction of $\rho,\rho'$ and Claim~\ref{cl:rho-rho'-disjoint} immediately imply the following.

\begin{claim}\label{cl:ex-disjoint-rho}
There exists disjoint rotations $\rho_1,\dots,\rho_k$ such that: for $j \in [k]$, $\rho_j$ is exposed in $T \cup \rho_j$; $M\setminus T = \cup_{j=1}^k \rho_j$.
\end{claim}


To conclude the proof that $\rho$ is a stitched rotation, first let 
$$\widetilde{M}=M\setminus (\cup_{j=1}^k \{xy| (x,y) \in \rho_j\}) \cup_{j=1}^k \{xy'|(x,y) \in \rho_j \hbox{ and } y'=f_{T^*}(x)\},$$
which is a matching by Claim~\ref{cl:rho-X-Y-disjoint} and Claim~\ref{cl:ex-disjoint-rho}. Recall that, for $j \in [k]$ and $(x_i,y+i), (x_{i+1},y_{i+1})$ consecutive pairs of $\rho_j$, we have $y_{i+1}=f_{T^*}(x_i)$, $x_i=\ell_{T^*}(y_{i+1})$, $y_i >_{x_i} y_{i+1}$, and $x_{i+1} >_{y_{i+1}} x$. In particular, for every agent $v$, we have $|r_{M \cup T^*}(v,M(v))-r_{M \cup T^*}(v,\widetilde M(v))| \leq 1$. We claim that $\widetilde{M}\in \Ss(T^*)$. By Claim~\ref{cl:ex-disjoint-rho}, $\widetilde M \subseteq T^*$. Suppose by contradiction that $xy\in T^*$ blocks $\widetilde{M}$, i.e., $ y >_x \widetilde{M}(x)$ and $ x >_y \widetilde{M}(y)$. Because $y_i >_{x_i} y_{i+1}$ and $x_{i+1} >_{y_i} x_i$ for all $j \in [k]$ and $(x_i,y_i), (x_{i+1},y_{i+1})$ consecutive pairs of $\rho_j$, $M$ and $\widetilde{M}$ do not block each other. In particular, $xy\not\in M$. Therefore, $x >_y\widetilde{M}(y)$, $y \neq M(y)$, and  $|r_{M \cup T^*}(y,M(y))-r_{M \cup T^*}(y,\widetilde M(y))| \leq 1$ imply $x >_y M(y)$. Similarly, $ y >_x M(x)$. Hence, $xy$ blocks $M$, contradicting the hypothesis. 

We can then apply Lemma \ref{lem:stitched_rotation_paird_matching} to deduce that $\rho$ is a stitched rotation w.r.t.\ $T^*$, concluding the proof of 1.

2. By Definition \ref{def:stitched_rotation}, $\rho$ is a non-singular rotation exposed in $T^*\cup \rho$. We can construct $M$ that is not blocked by any edge from $T^*$ through the following sequence of rotation eliminations, starting from $T^*\cup \rho$. By Lemma \ref{lem:singular_expose}, $\rho$ is not the only exposed rotation. Therefore we can keep on eliminating rotations other than $\rho$, until $\rho$ and $\overline{\rho}$ are the only two remaining rotations exposed. Eliminating $\overline{\rho}$ will leave $x_i$ and $y_i$ with a single entry in their preference list, see Lemma \ref{lem:only-rho-and-dual}. Hence, by Theorem \ref{thm:T_lem:eliminating-rotations-preserves-exposition}, there exists $M\in \Ss(T^*\cup \rho)$ with $\rho\subset M$. Since $T^*=E_S(T)$, we deduce that all edges of $T^*\cup \rho$ are stable, in $T^*\cup\rho$. In particular, $M$ is not blocked by any edge in $T^*$. Moreover, since all edges of $\rho$ are antipodal by definition, $\Ss(T^*)\subseteq \Ss(T^*\cup \rho)$, concluding the proof of 2. \end{proof}

\subsection{Deciding if a set of matchings is internally closed is co-NP-hard}\label{sec:NP_hard_deciding_internally_closed}

Using the structural properties we derived above, we now show that deciding if a set of matchings is internally closed is co-NP-hard through a reduction from 3-SAT. Recall that we called this problem \texttt{CIC}. Specifically, given any instance of 3-SAT, we construct a solvable instance $T$, and consider the problem of deciding if $\Ss(\widetilde{T})$ is internally closed, where $\widetilde{T}=E_S(T)$.

Let $\phi$ be any instance of 3-SAT with $n$ clauses, defined as: 
$$ \phi \quad \coloneqq \quad \bigwedge_{r=0}^{n-1}(h_{r,1}\vee h_{r,2}\vee h_{r3})
 \quad  = \quad \bigwedge_{r=0}^{n-1}C_{r},\quad\text{where }C_{r}=h_{r,1}\vee h_{r,2}\vee h_{r,3}
$$
We let each literal be included in the set $h_{r,i}\in\{v_{1},\overline{v_{1}},v_{2},\overline{v_{2}}\dots,v_k,\overline{v_k}\}$,
with a total of $k=O(n)$ variables. We say literal $h$ is \emph{positive} if $h=v_a$ and \emph{negative} if $h=\overline{v_a}$ for some variable $v_a$.

\paragraph{Construction of $T$ and $\widetilde T$.}
We construct an instance $T$ and subtable $\widetilde{T}$ in two steps.

\paragraph{Step 1: Construction of $\widetilde{T}$.}

\subparagraph{Step 1a: Construction of table $T_j$ for $j \in [k]$.}
For $j \in [k]$, let $cnt(v_{j})$ be the total number
of times variable $v_{j}$ (in the form of $v_{j}$ or $\overline{v_{j}}$)
appears in $\phi$, and let $n_{j}=2\cdot cnt(v_{j})-1$. For $j \in [k]$,
construct a roommate instance $T_{j}$ over variables $\{z_{j,i}\}_{i=1}^{n_j}\cup \{w_{j,i}\}_{i=1}^{n_j}$ as in Figure~\ref{fig:coNPH:instance}, left. Define:
\begin{align*}
\rho_{j} & =(z_{j,1},w_{j,1}),(z_{j,2},w_{j,2}),\dots,(z_{j,n_{j}},w_{j,n_{j}}),\\
\overline{\rho_{j}} & =(w_{j,2},z_{j,1}),(w_{j,3},z_{j,2}),\dots,(w_{j,1},z_{j,n_{j}}).
\end{align*}

\begin{figure}
\begin{center}
\begin{tabular}{||c|c||}
\hline 
Agent  & Preference list \tabularnewline
\hline 
$z_{j,1}$  & $w_{j,1}$ $w_{j,2}$\tabularnewline
\hline 
$z_{j,2}$  & $w_{j,2}$ $w_{j,3}$\tabularnewline
\hline 
$\vdots$ & $\vdots$\tabularnewline
\hline 
$z_{j,n_{j}}$  & $w_{j,n_{j}}$ $w_{j,1}$\tabularnewline
\hline 
$w_{j,1}$  & $z_{j,n_{j}}$ $z_{j,1}$\tabularnewline
\hline 
$\vdots$ & $\vdots$\tabularnewline
\hline 
$w_{j,n_{j}}$  & $z_{j,n_{j}-1}$ $z_{j,n_{j}}$\tabularnewline
\hline 
\end{tabular}
\hspace{.5cm}
\begin{tabular}{||c|c||}
\hline 
Agent  & Preference list \tabularnewline
\hline 
$z'$ & $w'$\tabularnewline
\hline 
$w'$ & $z'$\tabularnewline
\hline 
\end{tabular}
\par\end{center}
\caption{On the left: table $T_j$ for $j \in [k]$. On the right: table $T_0$.}\label{fig:coNPH:instance}
\end{figure}

Next lemma easily follows from the definition of exposed rotations. 

\begin{lemma}\label{rotations-of-Tj}
Let $j \in [k]$. Then, $\rho_j$ and $\overline{\rho_j}$ are the unique rotations exposed in $T_j$, and they are dual to each other. Moreover, $T_j = \rho_j \cup \overline{\rho_j}$.
\end{lemma}

\subparagraph{Step 1b: Construction of $T_{0}$.} We let $T_0$ be the roommate instance over agents $z',w'$ defined as in Figure~\ref{fig:coNPH:instance}, right. 

\subparagraph{Step 1c: Juxtaposition.}
Let $\widetilde{T}$ be the \emph{juxtaposition} of $T_{0}$ and all $T_{j}$ for $j \in [k]$. That is, the set of agents of $\widetilde{T}$ is the union of the agents from $T_0,T_1,\dots, T_k$, and their preference lists in $\widetilde{T}$ is exactly the preference lists in the roommate instance they belong to (note that $T_0$ and the $T_j$ are defined over disjoint sets of agents). 

\begin{lemma}\label{lem:NPH-internally-stable}
${\cal Z}(\widetilde T)={\cal Z}_{ns}(\widetilde T)=\{\rho_1,\overline{\rho_1}, \dots, \rho_k,\overline{\rho_k}\}$, $z'w'$ is the only fixed pair of $\widetilde{T}$, and $E_S(\widetilde{T})=\widetilde{T}$ (i.e., $\widetilde{T}$ is stable). 
\end{lemma} 
\begin{proof}
The first statement follows immediately from Lemma~\ref{rotations-of-Tj} and the fact that the sets of agents of $T_0,T_1,\dots,T_k$ are pairwise disjoint. The second is immediate. The third follows from the previous two and Lemma~\ref{lem:stable-pairs}.
\end{proof}


\paragraph{The function $A(\cdot )$.}
We now associate each literal $h_{r,i}$ with the pair $(v_{a},b)$, where variable $v_{a}$ corresponds to $h_{r,i}$ (in the form of $v_{a}$ or $\overline{v_{a}}$), and $b$ counts the occurrences of variable $v_a$ (as $v_{a}$ or $\overline{v_{a}}$) in $\phi$ as literal until $h_{r,i}$, including $h_{r,i}$. \\
 Let $A$ be the function that maps each literal $h_{r,i}$ to an agent as follows:
\[
A(h_{r,i})=\begin{cases}
z_{a,(2b-1)} & \text{if }h_{r,i}=\overline{v_{a}}\\
w_{a,(2b-1)} & \text{if }h_{r,i}=v_{a}
\end{cases}
\]
Note that $A(\cdot) $ is injective. We also observe the following.

\begin{lemma}\label{lemma:A-X-disjoint}
The images of functions $A(\cdot)$ and $\ell_{\widetilde T}(A(\cdot))$ are disjoint.
\end{lemma}
\begin{proof}
First observe that, if $A(h_{i,j})=w_{q,k}$ or $A(h_{i,j})=z_{q,k}$, then $k$ is odd. Therefore, by construction, all agents $w_{p,t}$ that are in the image of $\ell_{\widetilde T}(A(\cdot))$ have $t$ even, hence they are not in the image of $A(\cdot)$. Now consider an agent $z_{p,t}$ in the image of $\ell_{\widetilde T}(A(\cdot))$. By construction, $\ell_{\widetilde T}(w_{p,t})=z_{p,t}$. On the other hand, at most one of $z_{p,t}$ and $w_{p,t}$ is in the image of $A(\cdot)$, and the thesis follows.  \end{proof}

\paragraph{Step 2: Construction of $T$.} We now construct the roommate instance $T$ by adding to $\widetilde{T}$ certain antipodal edges. Hence, these edges will be of the form $xz$ with $z >_x f_{\widetilde{T}}(x)$ and $x <_z \ell_{\widetilde{T}}(z)$. For simplicity of exposition, we only present the addition of the edge to the row of $x$ or to the row of $z$, since the addition of the edge to the missing row is automatically implied by the antipodality. Moreover, unless otherwise specified, the relative order of all edges we add does not matter (i.e., if we add edges $xz$ and $xy$ with $x$ preferring $z,y$ to $f_{\widetilde T}(x)$, then we can arbitrarily have $z <_x y$ or $y <_x z$). 

\paragraph{Step 2a: Addition of antipodal edges.} 
For all $r \in [n-1]_0$, add $A(h_{r+1,1}),A(h_{r+1,2}),A(h_{r+1,3})$
to the lists of $\ell_{\widetilde T}(A(h_{r,1}))$, $\ell_{\widetilde T}(A(h_{r,2}))$, and
$\ell_{\widetilde T}(A(h_{r,3}))$ before entries in $\widetilde{T}$, where the first index of all literals is taken modulo $n$.

\paragraph{Step 2b: Adjust $T$ so that $E_S(T)=\widetilde{T}$.}
\begin{itemize}
\item Add each agent $a$ that is incident to any antipodal edge added
so far to the list of $z'$, such that $w'$ is the last preference
of $z'$.
\item For each agent $a$ that is incident to any of the antipodal edges added so far, add $z'$
as the first agent on its preference list after $\widetilde{T}$, that is,
$r(a,z')=r(a,\ell_{\widetilde{T}}(a))+1$. 
\end{itemize}

\begin{lemma}
$\widetilde{T}$ is the set of all stable edges of $T$, i.e., $\widetilde{T}=E_S(T)$.
\end{lemma}

\begin{proof}
 We know by Lemma~\ref{lem:NPH-internally-stable} that $\widetilde T$ is a stable table. It suffices to prove that all antipodal edges added in Step 2a and Step 2b to $\widetilde{T}$ are not stable edges of $T$. Since all edges we added are antipodal, none of them blocks any matching in $\mathcal{S}(\widetilde{T})$. Hence, $\widetilde{T}\subseteq E_S(T)$. 

Let us conversely show that $E_S(T)\subseteq \widetilde{T}$. Consider first  edges added in Step 2b. Since $w'$ has only $z'$ on its list and $w'z' \in \widetilde{T}$, we deduce that any matching from ${\cal S}(T)$ matches $w'$ to $z'$. Hence, any matching that matches $z'$ to any other partner would necessarily be unstable. Observe that all edges added in Step 2b involve $z'$ (but not $w'$), showing that none of them belongs to $E_S(T)$. 

It remains to consider edges added in Step 2a. Suppose for contradiction that a stable matching $M$ contains some antipodal edge $pq$ with $p,q\neq z'$, where $q$ lies before $\widetilde{T}$ on $p$'s list (and $p$ lies after $\widetilde{T}$ on $q$'s list). Since $z'w'$ is a fixed edge, therefore $z'w'\in M$. However, $z'$ prefers $q$ to $w'$, and $q$ prefers $z'$ to $p$, so $z'q$ would block $M$, which contradicts the assumption that $M$ is stable. Therefore none of the antipodal edges is a stable edge, and we conclude that $\widetilde{T}=E_S(T)$.
\end{proof}

\paragraph{Concluding the proof.} We now show that $\phi$ is satisfiable if and only if there is a stitched rotation wrt $\widetilde T$. Because of Theorem~\ref{augment_dual},  Theorem~\ref{thm:coNp-hard-check-closedness} follows.

\begin{lemma}\label{lem:if 3-SAT then IC}
If $\phi$ is satisfiable, then there exists a stitched rotation $\rho$
with respect to $\widetilde{T}$.
\end{lemma}

\begin{proof}
Let $\vec{v}$ be any valid assignment to $\phi$. Construct mapping
$f$ from $\vec{v}$ to a set of ordered pairs $\rho$ of edges from $T$ as follows. First, for all $r \in [n-1]_0$, 
pick any $j_r \in \{1,2,3\}$ such that $h_{r,j_r}$ evaluates to true under assignment $\vec{v}$. 
Then set
$$\rho=  (\ell_{\widetilde T}(A(h_{n-1,j_{n-1}})),A(h_{0,j_0})), (\ell_{\widetilde T}(A(h_{n-2,j_{n-2}})),A(h_{n-1,j_{n-1}})), \dots, (\ell_{\widetilde T}(A(h_{0,j_{0}})),A(h_{1,j_1})).$$
We claim that $\rho$ is a stitched rotation with respect to $\widetilde{T}$.




We first claim that each agent appears in at most one edge from $\rho$. Indeed, $A(\cdot)$ is injective by construction, while $\ell_{\widetilde T}(A(\cdot))$ is injective since by Lemma~\ref{lem:NPH-internally-stable} $\widetilde{T}$ is a stable table, and by Lemma~\ref{lem:T-star-is-valid} stable tables are valid. Hence, an agent appears in more than one edge if and only if $A(h_{r,j_r})=\ell_{\widetilde T}(A(h_{\overline{r},j_{\overline{r}}}))$ for some $r, \overline r \in [n-1]_0$. However, by Lemma~\ref{lemma:A-X-disjoint}, the images of $\ell_{\widetilde T}(A (\cdot))$ and $A(\cdot)$ are disjoint, and the claim follows.  

We next prove $\rho$ is a rotation exposed in $T'=\widetilde{T}\cup \{\rho\}$. 
By construction, $\ell_{\widetilde T} (A(h_{{r-1},j_{r-1}}),A(h_{r,j_r}))$ is an antipodal edge and $A(h_{r,j_r})$ is added in the list of $\ell_{\widetilde T} (A(h_{{r-1},j_{r-1}}))$ before entries from $\widetilde T$ (see Step 2a). Since we argued above that no other antipodal edge incident to $\ell_{\widetilde T}(A(h_{{r-1},j_{r-1}}))$ appears in $\rho$, we conclude that $f_{T'}(\ell_{\widetilde T}(A(h_{{r-1},j_{r-1}})))=A(h_{r,j_r})$. A symmetric argument lets us conclude that $\ell_{T'}(A(h_{r,j_r}))=\ell_{\widetilde T}(A(h_{{r-1},j_{r-1}}))$. Moreover,
$$s_{T'}(\ell_{\widetilde T}(A(h_{{r-1},j_{r-1}})))= f_{\widetilde T}(\ell_{\widetilde T}(A(h_{r-1,j_{r-1}}))=A(h_{r-1,j_{r-1}})$$
by validy of $\widetilde T$ (Lemma~\ref{lem:T-star-is-valid}). We conclude that $\rho$ is a rotation exposed in $\widetilde{T} \cup \{\rho\}$. 

In order to show that $\rho$ is a stitched rotation with respect to $\widetilde{T} \cup \{\rho\}$, by Lemma~\ref{lem:stitched_rotation_paird_matching} it suffices to show that there exists $M\in\Ss(\widetilde{T})$ such that $\ell_{\widetilde T}(A(h_{{r-1},j_{r-1}})) A(h_{r-1,j_{r-1}}) \in M$ for all $r \in [n-1]_0$.

Recall that, from Lemma~\ref{lem:NPH-internally-stable}, all rotations of $\widetilde{T}$ are non-singular and disjoint,
and $z'w'$ is the only fixed pair of $\widetilde{T}$. Using Theorem~\ref{thm:T_lem:eliminating-rotations-preserves-exposition} we deduce that, given $M\in {\cal M}(T)$, $M$ is stable if and only if $w'z' \in M$ and, for each $j \in [k]$, one of $\rho_j,\overline{\rho_j}$ is contained in $M$, and the other is disjoint from $M$.

\begin{claim}\label{cl:Y-set}
Fix $j \in [k]$. Then either $Y(\rho) \cap Y(\rho_j)=\emptyset$, or $Y(\rho)\cap Y(\overline{\rho_j})=\emptyset$.
\end{claim}

\noindent \underline{Proof of Claim}. Notice that $Y(\rho_j)=\{w_{j,1}, \dots, w_{j,n_j}\}$ and $Y(\overline{\rho_j})=\{z_{j,1}, \dots, z_{j,n_j}\}$. On the other hand, $Y(\rho)=\{A(h_{r,j_r})\}_{r \in [n-1]_0}$, where, for $r \in [n-1]_0$, $h_{r,j_r}$ is a literal that evaluates to true in $\vec{v}$. By definition of $A(\cdot)$, either $w_{j,q} \notin Y(\rho)$ for every $q$, or $z_{j,q} \notin Y(\rho)$ for every $q$. The thesis follows. 
$\hfill \diamond$

\smallskip 

By Claim~\ref{cl:Y-set} we can define for each $j \in [k]$ a rotation $\rho^\emptyset_j \in \{\rho_j,\overline{\rho_j}\}$ such that $Y(\rho^\emptyset_j) \cap Y(\rho)= \emptyset$. Let $\rho^*_j = \{\rho_j,\overline{\rho_j}\} \setminus \{\rho^\emptyset_j\}$, and consider $M= \{w'z'\} \uplus_{j=1}^k \rho^*_j $, which is a stable matching because of Theorem~\ref{thm:T_lem:eliminating-rotations-preserves-exposition} as discussed above. 
Let $r \in [n-1]_0$. If $A(h_{r,j_{r}})=z_{j,q}$ for some $j \in [k]$ and $q \in \mathbb{N}$, then $\rho^*_j=\overline{\rho_j}$ and $\ell_{\widetilde T}(A(h_{r,j_{r}})) A(h_{r,j_{r}})=w_{j,q+1} z_{j,q}\subseteq \rho^*_j$. Similarly, if $A(h_{r,j_{r}})=w_{j,q}$, we have $\ell_{\widetilde T}(A(h_{r,j_{r}})) A(h_{r,j_{r}})=z_{j,q} w_{j,q}\subseteq \rho^*_j=\rho_j$, concluding the proof.\end{proof}

\begin{lemma}\label{lem:if IC then 3-SAT}
If there exists $T'\supset \widetilde{T}$ as defined in the output of \texttt{CIC}($\widetilde{T}, T$), then $\phi$ is satisfiable. 
\end{lemma}

\begin{proof}
Let $M \in {\cal S}(T')\setminus \Ss(\widetilde T)$. By Theorem~\ref{augment_dual}, there exists a rotation $\rho$ stitched to $\widetilde T$.



\begin{claim}\label{cl:rho-all-clauses} 
For appropriate indices $j_0,\dots, j_{n-1}$, we have $$\rho = (\ell_{\widetilde T}(A(h_{n-1,j_{n-1}})), A(h_{0,j_0})), \ (\ell_{\widetilde T}(A(h_{n-2,j_{n-2}})), A(h_{n-1,j_{n-1}})),  \dots, (\ell_{\widetilde T}(A(h_{0,j_0})), A(h_{1,j_1})).$$ 
\end{claim} 
\noindent \underline{Proof of Claim}. $\rho \cap \widetilde T=\emptyset$ by definition of stitched rotation. Hence, $\rho$ is made of pairs added during Step 2a, that is, pairs from $\rho$ are of the form $(\ell_{\widetilde T}(A(h_{i-1,j})), A(h_{i,j'}))$ for some indices $i,j,j'$. Let $(x_{1},y_{1})=(\ell_{\widetilde T}(A(h_{i-1,j})), A(h_{i,j'}))$. We claim that the edge $(x_{0},y_{0})$ that precedes $(x_{1},y_{1})$ in $\rho$ is $(\ell_{\widetilde T}(A(h_{i,j'})), A(h_{i+1,j''}))$ for some $j''$. The statement then follows inductively. 

First observe that $$ A(h_{i,j'})=s_{\widetilde T \cup \rho}(x_{0}) = f_{\widetilde T}(x_{0})$$
by definition of rotation and since $\rho$ is a rotation exposed in $\widetilde T \cup \rho$ by definition of stitched rotation. By Lemma~\ref{lem:T-star-is-valid}, $\widetilde T$ is a valid table, hence $x_0=\ell_{\widetilde T}(A(h_{i,j'}))$. Using again the fact that $x_0 y_0$ was added during Step 2a, we deduce $y_0=A(h_{i+1,j''})$ for some index $j''$, as required. $\hfill \diamond$

\smallskip

Following Claim \ref{cl:rho-all-clauses}, we construct the following assignment of variables $\vec{v}=\{v_a\}_{a \in [k]}$: for each variable $v_a$, if $v_a = h_{r,j_r}$ for some $r,j_r$ such that $(\ell_{\widetilde{T}}(A(h_{r-1,j_{r-1}})),A(h_{r,j_r})) \in \rho$, set $v_a$ to true; else, set $v_a$ to false.  It suffices to show that the mapping above defines a true / false assignment for each variable $v_a$ satisfying the input 3-SAT formula $\phi$. Recall that $A$ is an injective function, so each $A(h_{r,j_r})$ in $\rho$ corresponds to a unique literal $h_{r,j_r}$. We start by observing the following.

\begin{claim}\label{cl:assignment-valid}
There does not exist $(\ell_{\widetilde{T}}(A(h_{r-1,j_{r-1}})),A(h_{r,j_r})), (\ell_{\widetilde{T}}(A(h_{r'-1,j_{r'-1}})),A(h_{r',j_{r'}})) \in \rho$ such that $h_{r,j_r}=v_a$ and $h_{r',j_{r'}}=\overline{v_a}$.
\end{claim}

\noindent \underline{Proof of Claim}. Suppose by contradiction the statement is false, 
 let $h_{r,j_r}$ be the $b$-th appearance of $v_a$ (counting $v_a$ or $\overline{v_a}$) in $\phi$, and $h_{r,j_r}$ be the $b'$-th appearance. Then by definition of function $A$,  $A(h_{r,j_r})=w_{a,2b-1}$ and $A(h_{r',j_{r'}})=z_{a,2b'-1}$.

Recall that $\rho$ is a stitched rotation w.r.t.~$\widetilde T$. By Lemma \ref{lem:stitched_rotation_paird_matching}, there exists a matching $M'\in \Ss(\widetilde{T})$ such that $\ell_{\widetilde T}(w_{a,2b-1})w_{a,2b-1}=z_{a,2b-1}w_{a,2b-1}\in M'$ and $\ell_{\widetilde T}(z_{a,2b'-1})z_{a,2b'-1}=w_{a,2b'}z_{a,2b'-1}\in M'$. However, by Lemma~\ref{rotations-of-Tj} and Theorem~\ref{thm:T_lem:eliminating-rotations-preserves-exposition}, any matching from ${\cal S}(\widetilde T)$ contains one of $\{z_{a,1}w_{a,1},z_{a,2}w_{a,2},$ $\dots,z_{a,n_a}w_{a,n_a}\}$ and $\{w_{a,1}z_{a,n_1},w_{a,2}z_{a,1},\dots,w_{a,n_a}z_{a,n_a-1}\}$ and is disjoint from the other, leading to the required contradiction. $\hfill \diamond$

\smallskip

To conclude the proof of the lemma, observe that the assignment that we constructed satisfies the $3$-SAT instance, since by Claim~\ref{cl:assignment-valid}, at least one literal per clause is set to true. \end{proof}

\section{On vNM stable sets in roommate instances}\label{sec:vNM_is_NP_hard}
In this section we show the implication of our results on internal closedness and for vNM stability. We show that the instance we constructed to prove coNP-hardness of \texttt{CvNMS} directly implies coNP-hardness of \texttt{CvNMS} and \texttt{FvNMS}. 

\subsection{\texttt{CvNMS} is co-NP-hard}

For any 3-SAT instance $\phi$, we follow the exact same construction of $\widetilde{T}$ and $T$ as in Section \ref{sec:NP_hard_deciding_internally_closed}. Recall that $\widetilde{T}$ consists exactly of all stable edges of $T$, i.e. $\widetilde{T}=\{M:M\in\Ss(T)\}$.

\begin{lemma}\label{cl:vNM iff IC}
${\cal S}(\widetilde{T})$ is vNM stable if and only if $\Ss(\widetilde{T})$ is internally closed.
\end{lemma}
\begin{proof}

We first observe an additional property of the roommate instance $T$.

\begin{claim}
$\Ss(\widetilde{T})=\Ss(T)$.
\end{claim}
\noindent \underline{Proof of Claim}. Since $\widetilde{T}=E_S(T)$, we have $\Ss(\widetilde{T})\supseteq \Ss(T)$. It suffices to show that for any edge in $e\in T\setminus \widetilde{T}$, $e$ does not block any matching in $\Ss(\widetilde{T})$. This holds since edges added in step 1 consists of exactly table $\widetilde{T}$; while any edge added in Step 2 are antipodal edges wrt $\widetilde{T}$ therefore cannot block any matching contained in $\Ss(\widetilde{T})$.  $\hfill \diamond$

\smallskip

By Lemma \ref{lem:vNM=IC+ES}, $\Ss(\widetilde{T})$ is vNM stable in $T$ if and only if $\Ss(\widetilde{T})$ is internally closed and $\Ss(\widetilde{T})$ is externally stable. $\Ss(\widetilde{T})$ is externally stable if and only if there does not exist $M\in \M(T)\setminus \Ss(\widetilde{T})$ such that $M$ is not blocked by $\Ss(\widetilde{T})$. Given that $\Ss(\widetilde{T})=\Ss(T)$,  
$\Ss(\widetilde{T})$ is externally stable if and only if there does not exist $M\in \M(T)\setminus\Ss(\widetilde{T})$ such that $\{M\}\cup \Ss(\widetilde{T})$ is internally stable, i.e., if and only if ${\cal S}(\widetilde T)$ is internally closed. 
\end{proof}

Since deciding whether ${\cal S}(\widetilde T)$ is internally closed is co-NP-hard, we conclude the following.

\begin{theorem}
\texttt{CvNMS} is co-NP-hard.
\end{theorem}

\subsection{vNM stable sets may not exist in a roommate instance}\label{sec:vNM not exist}
Before concluding our discussion about vNM stable matchings, we emphasize that in more general cases, vNM stable matchings in the roommate case is not guaranteed to exist. We can see by a simple example involving only 3 agents. 

\begin{example}\label{exp:no_LP_no_SM}
The inclusionwise maximal matchings in the roommate instance $T$ below are: $M_1=\{x_1x_2\},M_2=\{x_2x_3\},M_3=\{x_3x_1\}$. One can check that $M_1$ blocks $M_3$; $M_2$ blocks $M_1$; $M_3$ blocks $M_2$. 
\begin{center}
 T=\quad \begin{tabular}{||c c||} 
 \hline
 Person & Preference list \\ [0.5ex] 
 \hline
 $x_1$ & $x_2$ $x_3$   \\ 
 \hline
 $x_2$ & $x_3$ $x_1$   \\ 
 \hline
 $x_3$ & $x_1$ $x_2$ \\  [1ex] 
 \hline
\end{tabular}
\end{center}

By symmetry, without loss of generality, suppose there exists vNM stable set $\M\supset \{M_1\}$. Then $M_3\not\in \M$ because it is blocked by $M_1$. Then $M_2\in \M$ because $M_2$ is only blocked by $M_3$ which is not in $\M$. Yet $M_2$ blocks $M_1$, contradicting the definition of vNM stability. Therefore this instance has no vNM stable sets of matchings.\end{example}

In the instance from Example~\ref{exp:no_LP_no_SM}, no stable matching exists. It is not hard to show more involved examples where vNM stable sets may or may not exist in both solvable and non-solvable roommate instances, or where multiple vNM stable sets exist. We next investigate the problem of finding a vNM stable set in a roommate instance, or concluding that none exists.

\subsection{\texttt{FvNMS} is co-NP-hard}

\begin{theorem}\label{thm:vNMS_co-NP-Hard}
\texttt{FvNMS} is co-NP-hard.
\end{theorem}
\begin{proof}
From the reduction in Section \ref{sec:NP_hard_deciding_internally_closed}, we know that the following problem is  co-NP-hard: 
\begin{quote}
$(**)$ given a roommate instance $T$ with $\Ss(T)\neq \emptyset$, decide if $\Ss(E_S(T))$ is vNM stable. 
\end{quote}
To show that \texttt{FvNMS} is co-NP-hard, we reduce $(**)$ on input $T$ to \texttt{FvNMS} on input $T$. First observe that any $T'\subseteq T$ such that $\Ss(T')$ is vNM stable satisfies $T'\supseteq E_S(T)$. Moreover, if $T'\supsetneq E_S(T)$ is such that $\Ss(T')$ is vNM stable, then $T$ is a no-instance for $(**)$, since vNM stable sets are internally closed, hence inclusionwise maximal.

Assume first that \texttt{FvNMS}$(T)=T'$. Then $T$ is a yes-instance for $(**)$ if and only if $\Ss(T')=\Ss(E_S(T))$. The latter can be decided by computing the set of stable edges in $T'$ (see, e.g.,~\cite{Irving}) and checking if it coincides with $E_S(T)$. If conversely \texttt{FvNMS}$(T)= \text{null}$, then $T$ is a no-instance for $(**)$.

Therefore, an efficient algorithm to solve \texttt{FvNMS} can be transformed into an efficient algorithm to solve $(**)$. The thesis follows. 
\end{proof}

\section{Conclusions}\label{sec:conclusions}

In this paper we introduce a generalization of vNM stable matchings called internally closed matchings, given by inclusionwise maximal internally stable sets of matchings. This new class of problem is structurally interesting: in the bipartite case, internally closed sets exhibit a lattice structure which corresponds to the set of stable matchings in an easily characterized sub-instance of the original problem; in the non-bipartite case, characterizing internally closed sets is co-NP-hard for which we show a reduction from 3-SAT. co-NP-hardness of internally closed sets also naturally leads to an co-NP-hardness proof for deciding vNM stability in the roommate case. Our structural results, in turn, extend the classical theory of rotation posets for both the marriage and roommate problems.

Several interesting future lines of research exist. Our analysis of internal closedness and vNM stability in the roommate case is restricted to solvable instances. On the one hand, our restriction shows that deciding internal closedness and vNM stability is already co-NP-hard even when considering a subdomain that seems most tractable. On the other hand, it would be interesting to investigate properties of internally closed / vNM stable sets of matchings of unsolvable roommate instances.  In addition, while we have shown that \texttt{CIC}, \texttt{CvNMS}, and \texttt{FvNMS} are co-NP-hard, further questions remain about their complexities.  It is clear that \texttt{CvNMS} belongs to co-NP, since a ``no'' certificate is given by a matching $M \in \M(T)\setminus \Ss(\widetilde T)$ that is not blocked by $\widetilde T$. Whether \texttt{CIC} $\in$ co-NP remains open. The complexity of the problem of deciding whether a roommate instance has a vNM stable set is also open. Note that the latter problem is easier than \texttt{FvNMS}, since it only requires a yes/no answer, and no subtable as an output. 


\index{Bibliography@\emph{Bibliography}}%
\bibliographystyle{plain}
\bibliography{cite.bib}
\begin{appendices}
\section{Proof of Lemma \ref{symmetric_diff}}\label{appendix_sym_diff}
\begin{proof}
1. Because $M$ and $M'$ are matchings, each node in $G(M\triangle M')$ has degree at most 2, so connected components of $G(M\triangle M')$ can only be singletons or even cycles or paths. Suppose by contradiction that there is a path $P=x_0,x_1,...x_{k-1}$ ($k\leq 2$). Assume without loss of generality that $x_0x_1\in M$, then $x_0$ is unmatched in $M'$, therefore $x_0$ prefers $M$. Because $M$ and $M'$ do not block each other, $x_1$ must prefer $M'$ to $M$, so $M'(x_1)=x_2 >_{x_1} x_0$. Iterating through path $P$ we get $x_{i+1}>_{x_i} x_{i-1}$ for all $i\in [k-1]_0$ (subscript modulo $k$). Therefore for the last node $x_{k-1}$ we have $\emptyset >_{x_{k-1}} x_{k-2}$ otherwise $x_{k-2}x_{k-1}$ would be a blocking pair. However, $\emptyset <_x y$ for any $x\neq y$ such that $xy\in T$, so we get a contradiction. Therefore, $G(M\triangle M')$ contains only singletons or even cycles with size at least 4. It immediately follows that $M$ and $M'$ match the same set of nodes. 
    
2. From 1.\ we know that $M\cup M'$ contains only even cycles, singletons, or edges that are shared by $M$ and $M'$, because $M$ and $M'$ match the same set of nodes, both nodes in an irregular edge must be matched in both matchings. None of the common edges can be irregular. Therefore irregular edges can only be in even cycles in $G(M\triangle M')$. Assume by contradiction that $x_ix_{i+1}$ is an irregular edge in the even cycle $C=x_0,x_1,...,x_{2k-1}$, and both $x_i,x_{i+1}$ prefer $M$ to $M'$, then we must have $x_ix_{i+1}\in M'$, otherwise $x_ix_{i+1}\in M$ blocks $M'$. So $M(x_i)=x_{i-1}>_{x_i} x_{i+1}$ and $M(x_{i+1})=x_{i+2}>_{x_{i+1}} x_i$. By a similar argument as 1), in order for $x_{i+1}x_{i+2}\in M$ not to block $M'$, we must have $M'(x_{i+2})=x_{i+3}>_{x_{i+2}} x_{i+1}$ and through iteration, we have $x_{n+1}>_{x_n} x_{n-1}$ for any $n\in [2k-1]_0$ (subscript modulo 2k). Likewise, in order for $x_ix_{i-1}$ not to block $M'$, we must have $M'(x_{i-1})=x_{i-2}>_{x_{i-1}} x_i=M(x_{i-1})$, and through iteration, we have $x_{n-1}>_{x_n} x_{n+1}$ for any $n\in [2k-1]_0$ (subscript modulo 2k). This is a contradiction. Therefore there cannot exist any irregular edge.
\end{proof}

\end{appendices}

\end{document}